\documentclass[11pt,a4paper,leqno]{article}

\usepackage{amsmath,amsthm,amssymb,txfonts,fourier}
 \usepackage[pdftex]{graphicx}
\usepackage[latin1]{inputenc}
\numberwithin{equation}{section}
\usepackage{color}

\newcommand{\la}{\lambda}

\newcommand{\vfi}{\varphi}
\newcommand{\eps}{\varepsilon}

\newcommand{\tiv}{\tilde{v}}
\newcommand{\tiz}{\tilde{z}}
\newcommand{\tiw}{\tilde{w}}
\newcommand{\tiu}{\tilde{u}}
\newcommand{\tivfi}{\tilde{\vfi_{1}}}

\newcommand{\dys}{\displaystyle}
\newcommand{\ov}{\overline}

\def\R{\mathbb{R}}

\def\sob{H^{1}(\mathC)}
\def\huno{\mathcal{H}^{1}(\mathC)}
\newcommand{\elle}[1]{\mathcal{L}^{#1}(\Omega)}
\def\sobmezzo{H^{1/2}(\Omega)}
\def\hmezzo{\mathcal{H}^{1/2}(\Omega)}

\def\acca{\mathcal{H}}

\def\mathT{\mathcal{T}}
\def\mathS{\mathcal{S}}

\newcommand{\from}{\colon}
\newcommand{\slashint}{\int}
\newcommand{\intm}{\slashint}
\newcommand{\intc}{\int_{\mathC}}
\newcommand{\into}{\int_{\Omega}}

\newcommand{\pa}{\partial}

\newcommand{\mathC}{\mathcal{C}}
\newcommand{\tr}{\mathrm{tr}}
\newcommand{\lapneu}{(-\Delta_{\mathcal N})^{1/2}}
\newcommand{\mezzolap}{(-\Delta)^{1/2}}
\newtheorem{theorem}{Theorem}[section]
\newtheorem{remark}[theorem]{Remark}
\newtheorem{lemma}[theorem]{Lemma}
\newtheorem{proposition}[theorem]{Proposition}
\newtheorem{definition}[theorem]{Definition}
\newtheorem{corollary}[theorem]{Corollary}
\newcommand{\bos}{\begin{remark}\rm}
\newcommand{\eos}{\end{remark}}
\newcommand{\bte}{\begin{theorem}}
\newcommand{\ete}{\end{theorem}}
\newcommand{\bpr}{\begin{proposition}}
\newcommand{\epr}{\end{proposition}}
\newcommand{\ble}{\begin{lemma}}
\newcommand{\ele}{\end{lemma}}
\newcommand{\bdf}{\begin{definition}}
\newcommand{\edf}{\end{definition}}
\newcommand{\bdm}{\begin{displaymath}}
\newcommand{\edm}{\end{displaymath}}
\newcommand{\beq}{\begin{equation}}
\newcommand{\eeq}{\end{equation}}
\newcommand{\bdim}{\begin{proof}}
\newcommand{\edim}{\end{proof}}
\newcommand{\dyle}{\displaystyle}

\begin{document}

\title{Fractional diffusion with\\ Neumann boundary conditions:\\
the logistic equation}
\author{Eugenio Montefusco, Benedetta Pellacci \& Gianmaria Verzini}
\date{\today}
\maketitle

\begin{abstract}
Motivated by experimental studies on the anomalous diffusion of  biological populations, we introduce a nonlocal differential operator which can be interpreted as the spectral square root of the Laplacian in bounded domains with Neumann homogeneous boundary conditions. Moreover, we study related linear and nonlinear problems exploiting a local realization of such operator as performed in \cite{cata} for Dirichlet homogeneous data. In particular we tackle a class of nonautonomous nonlinearities of logistic type, proving some existence and uniqueness results for positive solutions by means of variational methods and bifurcation theory.
\end{abstract}

\section{Introduction}
Nonlocal operators, and notably fractional ones, are a classical topic
in harmonic analysis and operator theory, and they are recently becoming
impressively popular because of their connection with many real-world phenomena,
from physics \cite{mekl,GG,NT} to mathematical nonlinear analysis \cite{ABS,SV},
from finance \cite{A,CT} to ecology \cite{caro,rerh,Hu,beroro}.
A typical example in this context is provided by
L\'evy flights in ecology: optimal search
theory predicts that predators should adopt search strategies
based on long jumps --frequently called L\'evy flights-- 
where prey is sparse and distributed unpredictably,
Brownian motion being more efficient only for locating abundant
prey (see \cite{shzakl,Vo,Hu}).   As the dynamic of a population
dispersing via random walk is well described by a local
operator --typically the Laplacian-- L\'evy diffusion processes are generated
by fractional powers of the Laplacian $(-\Delta)^{s}$
for $s\in (0,1)$  in all $\R^{N}$. These operators in $\R^{N}$
can be defined equivalently in
different ways, all of them enlightening their nonlocal nature,
but, as shown in \cite{casi} and \cite{casasi},
they admit also local realizations: the fractional Laplacian
of a given function $u$ corresponds to the
Dirichlet to Neumann map of a suitable extension of $u$ to $\R^{N}\times[0,+\infty)$.
On the contrary, on bounded domains, different not equivalent definitions are available
(see e.g. \cite{guma,anmaro,cata} and references therein). This variety reflects
the different ways in which the boundary conditions can be understood in the
definition of the nonlocal operator. In particular, we wish to mention the recent
paper by Cabr\'e and Tan \cite{cata}, where the operator $\mezzolap$
on a bounded domain $\Omega\subset\R^N$ and associated to homogenous
Di\-ri\-chlet boundary conditions is defined  by Fourier series,
using a basis of corresponding eigenfunctions of $-\Delta$. Their point
of view allows to recover also in the case of a bounded domain the aforementioned 
local realization: indeed, interpreting $\Omega=\Omega\times\{0\}$ as a part
of the boundary of the cylinder $\Omega\times(0,+\infty)\subset \R^{N+1}$, the Dirichlet spectral 
square root of the Laplacian coincides with the Dirichlet to Neumann map for 
functions which are harmonic in the cylinder and zero on its lateral surface.
These arguments can be extended also to different powers
of $-\Delta$, see \cite{cadadusi}.
On the other hand, in  population dynamic, Neumann
boundary data are as natural as Dirichlet ones, as they represent
a boundary acting as a perfect barrier for the population.
The aim of this paper is then to provide a first contribution in the study
of the spectral square root of the Laplacian with Neumann boundary conditions.
\\
Inspired by \cite{cata}, our first goal is to provide
a formulation of the problem
\beq\label{linear}
\begin{cases}
(-\Delta)^{1/2} u= f & \hbox{in } \Omega, \\
\pa_{\nu} u=0 & \hbox{on } \pa\Omega,
\end{cases}
\eeq
where $\Omega$ is a $C^{2,\alpha}$ bounded domain
in $\R^{N}$, $N\geq 1$, and $f$ can be thought, for instance, as an $L^{2}(\Omega)$
function. To this aim, let us denote with $\left\{\phi_{k}\right\}_{k\geq 0}$ an
orthonormal basis in $L^2(\Omega)$ formed by eigenfunctions
associated to eigenvalues $\mu_{k}$
of the  Laplace operator subjected to homogenous Neumann
boundary conditions, that is
\beq\label{autof}
\begin{cases}
-\Delta \phi_{k}=\mu_{k}\phi_{k} & \hbox{ in }\Omega, \\
\partial_{\nu}\phi_{k}=0 & \hbox{ on }\pa\Omega.
\end{cases}
\eeq
We can define the operator $(-\Delta)^{1/2} \from H^{1}(\Omega)\to
L^{2}(\Omega))$
 by
\beq\label{deflap}
(-\Delta)^{1/2} u=\sum_{k=1}^{+\infty}\mu_{k}^{1/2}u_{k}\phi_{k}
\quad \hbox{for $u$ given by} \quad
u=\sum_{k=0}^{+\infty}u_{k}\phi_{k}.
\eeq
The first series in \eqref{deflap} starts from $k=1$
since the first eigenvalue and the cor\-re\-spon\-ding
eigenfunction in \eqref{autof} are given by $(\mu_{0},\phi_{0})=(0,1/\sqrt{|\Omega|})$.
This simple difference with the Laplacian subjected to homogeneous
Dirichlet boun\-da\-ry conditions has considerable effects.
First of all, this implies that $(-\Delta)^{1/2}$, as the usual Neumann Laplacian,
has a nontrivial kernel made of the constant functions,
then it is not an invertible operator and \eqref{linear}
cannot be solved without  imposing additional conditions
on the datum $f$; on the other hand, given any $u$ defined on $\Omega$, its 
harmonic extension on ${\mathcal C}:=\Omega\times (0,+\infty)$ having zero
normal derivative on the lateral surface needs not to belong to any Sobolev
space, as constant functions show. These features has to be taken into
account when establishing the functional framework
where to set the variational formulation of  \eqref{linear}.
In this direction, we will first provide a proper
interpretation of \eqref{linear}, and a corresponding local realization, in the zero mean setting.
To this aim, let us introduce the space of functions defined in the cylinder ${\mathcal C}$
$$
\huno:= \left\{ v\in\sob:\, \int_{\Omega} v(x,y)\, dx=0, \; \forall
\, y\in(0,+\infty) \right\}.
$$
An easy application of the Poincar\'e-Wirtinger inequality
shows that we can choose as a norm of $v\in\huno$  the $L^{2}$ norm
of the gradient of $v$ (see Proposition \ref{propo:mediacontinua}
and Lemma \ref{normahuno}).
It comes out that, when the datum $f$ has zero mean, a possible solution
of \eqref{linear} is the trace of a function belonging to
$\huno$.  The corresponding space of traces can be equivalently defined
in different ways, since Proposition \ref{prop:H1/2} shows
that
\begin{align*}
 \hmezzo &:= \left\{ u\in\sobmezzo:\, \int_{\Omega} u(x)\, dx=0\right\}
=\left\{ u=v(x,0):\, v\in\huno\right\}
\\
&=\left\{ u\in L^{2}(\Omega) :\, u=\sum_{k=1}^{+\infty} u_{k}\phi_{k}
\hbox{ such that }
\sum_{k=1}^{+\infty} \mu^{1/2}_{k}u^{2}_{k}<+\infty \right\}.
\end{align*}
In proving this result, one obtains that every $u\in \hmezzo$
has an harmonic extension $v\in\huno$ given by
\beq\label{armointro}
v(x,y)=\sum_{k=1}^{+\infty} u_{k}\phi_{k}(x)e^{-\mu^{1/2}_k y},
\qquad\text{for $(\mu_{k},\phi_{k})$ solving \eqref{autof}},
\eeq
and which is also the unique weak solution of the problem
\beq\label{pbestintro}
\begin{cases}
\Delta v=0 &\text{ in } \mathC,
\\
\pa_{\nu} v=0 &\text{ on } \pa \Omega\times (0,+\infty),
\\
v(x,0)=u(x) &\text{ on } \Omega.
\end{cases}
\eeq
Thus, given $u\in\hmezzo$ we can find a unique $v\in \huno$ solving
\eqref{pbestintro}, for which it is well defined the functional acting on
$\sobmezzo$ as
$$
\langle -\partial_y v(\cdot,0), g \rangle := \int_{\mathC} \nabla v \cdot \nabla \tilde g\,dxdy,
$$
where $\tilde g$ is any $H^{1}({\mathcal C})$ extension of $g$.
Since this functional is actually an ele\-ment of the dual
of $\hmezzo$, it is well defined the operator
$L_{1/2}(u)=-\partial_{y}v(\cdot,0)$ between $\hmezzo$ and its dual.
Thus, restricting the study to the zero mean function spaces, and
taking into account equations \eqref{deflap} and \eqref{armointro},
we have that $L_{1/2}$ conincides with
$(-\Delta)^{1/2}$,  but it is invertible: for every
$f$ in the dual space of $\hmezzo$ there exists a unique
$u\in\hmezzo$ such that
$L_{1/2} u=f$, and
this function $u$ is the trace on $\Omega$
of the unique solution $v\in \huno$ of the problem
(see Lemma \ref{lem:weak-v})
\beq\label{linvintro}
\begin{cases}
\Delta v=0 & \hbox{in }\Omega\times(0,+\infty), \\
\pa_{\nu} v=0 & \hbox{on } \pa\Omega\times(0,+\infty),\\
\pa_{\nu} v(x,0)= f(x) & \hbox{on } \Omega.
\end{cases}
\eeq
The link between $L_{1/2}$ and $(-\Delta)^{1/2}$ now becomes
 transparent since
$$
(-\Delta)^{1/2}(u)=L_{1/2}\left(u-\slashint u\right)
$$
that is, the image of a function $u$ trough $(-\Delta)^{1/2}$
is the same of the one yield by $L_{1/2}$ acting on the zero
mean component of $u$ (see Definition \ref{df:lmezzo}).
In this way we have recovered the local realization
of $(-\Delta)^{1/2}$ as a map Dirichlet-Neumann since
$$
(-\Delta)^{1/2}u=L_{1/2}\left(u-\slashint u\right)=
-\partial_{y}\tiv(x,0)=\partial_{\nu}\tiv
$$
where $\tiv$ solves (1.5) with Dirichlet datum $\tiu=u-\slashint u$
instead of $u$.
Therefore, if $f$ has zero mean, 
denoting with $\tiu(x)=\tiv(x,0)$ the unique solution of
$L_{1/2}\tiu={f}$ then the solutions set
of \eqref{linear} is given by $\tiu+h$ for $h\in \R$.
\\
Since we are interested in ecological applications, as a first study we focus
our attention on the logistic equation.
More precisely, consider a population dispersing via the above defined anomalous diffusion
in a bounded region $\Omega$, with Neumann boundary
conditions, growing logistically within the region;
then $u$, the population density, solves the diffusive equation
$$
\begin{cases}
u_{t}=d\mezzolap u(x,t)+u(x,t)(m(x)-u(x,t)) & \text{in } \Omega\times (0,+\infty),
\\
\partial_{\nu} u=0 & \text{on } \partial \Omega\times (0,+\infty),
\\
u(x,0)=u_{0}(x)& \text{in } \Omega,
\end{cases}
$$
where $d>0$ acts as a diffusion coefficient,
 the term $-u^{2}$ express the self-limi\-ta\-tion of the population
and $m\in C^{0,1}(\ov{\Omega})$ corresponds to
the birth rate of the population if self-limitation is ignored. The weight $m$
may be positive or negative in different regions, denoting favorable or hostile habitat, 
respectively. The stationary states of this equation are the solutions of the following
nonlinear problem
\beq
\begin{cases}\label{logim}
(-\Delta)^{1/2}u=\la u(m(x)-u) & \text{in } \Omega,
\\
\partial_{\nu} u=0 & \text{on } \partial \Omega,
\end{cases}
\eeq
where $\lambda=1/d>0$.
When the diffusion follows the rules of the Brownian motion
this model has been introduced in \cite{sk} and studied by
many authors (see \cite{caco2} and the references therein).
One of the major task in this problem is describing
how favorable  and unfavorable habitats, represented by
the interaction between $\la$ and $m$, affects the overall
suitability of an environment for a given populations
\cite{caco}.
The typical known facts for the stationary problem associated to
Brownian motion can be summarized as follows:
\begin{theorem}[\cite{he,um}]\label{eslap}
 i) If the function $m$ has negative mean inside $\Omega$ and it
is positive somewhere, then there exists a positive number
$\overline{\mu}_{1}$ such
that for every $\la>\overline{\mu}_{1}$ there exists a unique positive solution $u_{\la}$
\beq\label{deltalog}
\begin{cases}
-\Delta u=\la u(m(x)-u) & \text{in } \Omega,
\\
\partial_{\nu} u=0 & \text{on } \partial \Omega,
\end{cases}
\eeq
and $u_{\la}\to 0$ as $\la\to \overline{\mu}_{1}^{+}$.
\\
ii) If $m$ has nonnegative average, then for every $\la>0$
there exists a unique positive solution $u_{\la}$ of \eqref{deltalog}
and $u_{\la}\to h^{*}$
as $\la\to 0^{+}$, for $h^{*}$ expressed by
\beq\label{defhintro}
h^{*}=\slashint m(x)dx=\frac1{|\Omega|}\into m(x)dx.
\eeq
\end{theorem}
The number $\overline{\mu}_{1}$ appearing in i) is the first
positive eigenvalue with positive
eigenfunction of the  operator $-\Delta$ with Neumann boundary condition
and with a weight $m$ satisfying the hypotheses in i).
\\
In our situation, we have, first of all,
to clarify that by a weak positive solution of \eqref{logim} we mean
a function $u\in H^{1/2}(\Omega)$, $u(x)>0$,
$u(x)=\tiu(x)+h$ with $h\in \R^{+}$ and $\tiu\in \hmezzo$,
so that $\tiu(x)=\tiv(x,0)$ for $\tiv\in\huno$ and
$(\tiv,h)\in \huno\times \R$ is a
weak solution of the nonlinear problem
\beq\label{nonlinprobintro}
\begin{cases}
\Delta \tiv=0 & \hbox{in }\Omega\times(0,+\infty),
\\
\pa_{\nu} \tiv=0 & \hbox{on } \pa\Omega\times(0,+\infty), \\
\pa_{\nu} \tiv= \la(\tiv+h)(m(x)-\tiv-h) & \hbox{on } \Omega\times\{0\},
\\
\dys \into \la(\tiv(x,0)+h)( m(x)-\tiv(x,0)-h)dx=0,
\end{cases}
\eeq
in the sense that
$$
\begin{cases}
\dys\int_{\mathC}\nabla \tiv\nabla \psi\,dxdy= \into\la(\tiv+h)(m(x)-\tiv-h)\psi\,dx
\qquad \forall\,\psi\in \huno,
\medskip\\
\dys\into \la(\tiv+h)(m(x)-\tiv-h)dx=0.
\end{cases}
$$
In other words, we impose that the right hand side
has zero mean, choosing, in this way, the mean of a solution $u$ as
$h=h_{u}$. Then we obtain the well posedeness of the problem
$$
\begin{cases}
L_{1/2}\tiu=\la (\tiu+h_{u})(m(x)-\tiu-h_{u}) &\hbox{in } \Omega,
\\
\partial_{\nu}\tiu=0 & \hbox{on } \pa\Omega,
\end{cases}
$$
since now the right hand side has zero mean, and
we obtain in this way the zero part mean of $u$.
Moreover, notice that the mean of the function $v(x,y)$ solution
of \eqref{nonlinprobintro} with $v(x,y)=\tiv+h_{u}$ and
$\tiv(x,0)=0$  is exactly the mean of $u$.
\\
Our main existence result is the following
\bte\label{esintro}
Let $m\in C^{0,1}(\ov{\Omega}),\, m\not\equiv 0$. Then the following conclusion hold:
\\
i) If $\into m(x)dx<0 $ and there exists $x_{0}\in \Omega$
such that $m(x_{0})>0$, then there exists a positive number
$\la_{1}$ such that for every $\la>\la_{1}$ there exists a unique positive solution $u_{\la}$ of \eqref{logim}, with $u_{\la}\to 0$ as
$\la\to \la_{1}^{+}$ .
\\
ii) If $\into m(x)dx\geq0$, then for every $\la>0$ there exists a unique
positive solution $u_{\la}$ of \eqref{logim} with $u_{\la}\to h^{*}$ for
$\la\to 0^{+}$.
\ete
As in the standard diffusion case, $\la_{1}$ is the first positive eigenvalue  with positive eigenfunction of the problem
$$
\begin{cases}
\mezzolap u=\la m(x)u &\hbox{in }\Omega,
\\
\partial_{\nu}u=0 &\hbox{on } \pa\Omega.
\end{cases}
$$
which existence is proved in Theorem \ref{primoautopos}. Theorem
\ref{esintro} will be obtained via classical bifurcation theory, indeed,
in case {\it i)}, we can show that a smooth
cartesian branch of positive solutions
bifurcates from the trivial solution $(\la,h,\tiu)=(\la_{1},0,0)$,
this branch can be
continued  in all the interval $(\la_{1},+\infty)$, and contains all the positive
solutions of 	\eqref{logim}, that is to say that for every $\la>0$ there exists a unique positive solution (see Proposition \ref{bifurcation1}, and
Theorem\ref{global}).  We tackle case {ii)} first assuming that
the mean of $m$ is positive.
This allows us choose as a bifurcation
parameter $h$, the future mean of $u$, instead of $\la$,
and find a branch bifurcating from the trivial solution
$(\la,h,\tiu)=(0,h^{*},0)$, with $h^{*}$ defined as
in \eqref{defhintro}.
As in the previous case we can show that this branch is global and contains
all the positive solutions (see Proposition \ref{bifurcation2}, and
Theorem\ref{global}).
Finally, we complete the proof of  case {\it ii)} by approximation in Theorem
\ref{media zero}.
\\
All the effort made in finding the proper formulation
for the linear and the nonlinear problem enables us to
prove the existence results for \eqref{logim}, which
are in accordance with the case of standard diffusion. But,
trying to  enlighten the differences between the two models,
one has to take care of the eigenvalues appearing in Theorems
\ref{eslap} and \ref{esintro}, that is $\overline{\mu}_{1}$ and
$\la_{1}$. Since such eigenvalues act as a survival threshold in hostile habitat, 
it is a natural question to wonder which is the lowest one, indeed
this indicates whether or not the fractional search strategy is preferable 
with respect to the brownian one. This appears to be a difficult question,
since the eigenvalues depend in a nontrivial way on $m$, and also on the sequence 
$(\mu_k)_k$ defined in \eqref{autof}. At the end of Section \ref{sec:logistica}
we report some simple numerical experiments to hint such complexity. 

\section{Functional setting}\label{sec:funct_setting}
In this section we will introduce the functional spaces where the spectral
Laplacian associated to homogeneous Neumann boundary
conditions will be defined.
Moreover, we will study the main properties of this operator and
find the proper conditions under which the inverse operator is well
defined. Finally, we will prove summability and regularity properties
enjoyed  by the solutions of the linear problem.
\\
Throughout the paper $\Omega$ is a $C^{2,\alpha}$ bounded domain and
we will use the notation ${\mathcal C}=\Omega\times (0,+\infty)$.
\\
In this plan we will make use of the following projections operators.
\begin{definition}\rm
Let us define the operators
$ A_\mathC,\, Z_\mathC \from \sob \to \sob $ by
\[
A_\mathC v :=\slashint_\Omega v(x,\cdot)\,dx=\frac1{|\Omega|}\into
v(x,\cdot)dx,\qquad Z_\mathC v := v- A_\mathC v,
\]
for $|\Omega|$ denoting the Lebesgue measure of the domain $\Omega$.
$A_{\mathC}$ and $Z_{\mathC}$
give the ave\-rage (with respect to $x$) and the zero-averaged part of
a function $v$, respectively. Analogously, for $u\in H^{1/2}(\Omega)$, we write
\beq\label{zomega}
A_\Omega u :=\slashint_\Omega u(x)\,dx,\qquad Z_\Omega u :=u - A_\Omega u.
\eeq
When no confusion is possible, we drop the subscript
in $A$, $Z$.
\end{definition}
It is standard to prove that, in both cases, $A$ and $Z$ are linear and continuous,
and that $\tr_\Omega \circ Z_\mathC = Z_\Omega \circ \tr_\Omega$. Since the
integration in the definition  of $A_{\mathC}$ is performed
only with respect to the $x$ variable, it is natural to interpret
the image of a function $v$ through the operator $A_{\mathC}$ as a function of one variable.
$A_{\mathC}v(y)$ enjoys the following properties.
\bpr\label{propo:mediacontinua}
If $v\in\sob$ then $A_\mathC{v}\in H^{1}(0,+\infty)$.
In particular, it is a continuous function up to $0^+$,
and it vanishes as $y$ tends to infinity.
\epr
\bdim
Since $\pa_{y}v(\cdot, y)\in L^{2}(\Omega)$ for almost every
$y$, we can compute
$(A_\mathC v)'(y)$ and obtain, by H\"older's ine\-quality,
\begin{align*}
\int_0^{+\infty} \left((A_\mathC v)'\right)^2\,dy
&= \int_0^{+\infty} \frac{1}{|\Omega|^2}\left(\into \pa_{y}v\,dx\right)^2\,dy
\\
&\leq \int_0^{+\infty}\frac{1}{|\Omega|} \left(\into |\pa_{y}v|^2\,dx\right)\,dy<+\infty.
\end{align*}
As a consequence, $A_\mathC{v} \in H^1(0,\infty)$,
so that it is continuous in $y$ and it vanishes as $y$ tends to $+\infty$.
\end{proof}
Introducing  the following functional spaces
\beq\label{defhuno}
\begin{split}
\huno&:=\mathrm{Ker}A_\mathC = \left\{ v\in\sob:\, \int_{\Omega} v(x,y)\, dx=0, \; \forall
\, y\in(0,+\infty) \right\},\\
\hmezzo &:=\mathrm{Ker}A_\Omega = \left\{ u\in\sobmezzo:\, \int_{\Omega} u(x)\, dx=0\right\},
\end{split}
\eeq
it is worth noticing that the former is  well defined by Proposition \ref{propo:mediacontinua}. Moreover, we can choose
as a norm on $\huno$ the quantity
\begin{equation}\label{normacal}
\|v\|_{\huno}^{2}:=\|\nabla v\|_{L^{2}}^{2}
\end{equation}
as it is equivalent to the $H^1$-norm thanks to the following lemma.
\ble\label{normahuno}
There exists a positive constant $K$ such that for every $v\in\huno$
it holds
\bdm
\|v\|_{L^{2}} \leq K \|\nabla v \|_{L^{2}}.
\edm
\ele

\bdim
We set $\nabla_{\!\! x} v=(\pa_{x_{1}}v,\dots,\pa_{x_{n}}v)$ and we notice that
for any $v\in\huno$ the Poincar\'e-Wirtinger inequality implies
\bdm
\begin{array}{rl}
\dys\|v\|_{L^{2}}^{2}& \dys=\int_{0}^{+\infty}dy \int_{\Omega} v^{2}(x,y) dx
\leq \int_{0}^{+\infty}\left( c_{pw} \int_{\Omega} |\nabla_{x} v(x,y)|^{2} dx\right)dy
\\ \\
&\dys \leq c_{pw} \intc |\nabla_{x} v(x,y)|^{2} dx dy
\leq c_{pw} \intc |\nabla v(x,y)|^{2} dx dy=K^{2} \|\nabla v \|_{L^{2}}^{2}
\end{array}
\edm
proving the claim.
\edim
The following proposition gives a complete description of the space $\hmezzo$.
\bpr\label{prop:H1/2}
Let $\hmezzo$ be defined in \eqref{defhuno}. Then
the following conclusions hold:
\begin{align*}
\text{(i)}\quad \hmezzo
&=\left\{ u=\mathrm{tr}_\Omega v:\, v\in\huno\right\}
\\
&=\left\{ u\in L^{2}(\Omega) :\, u=\sum_{k=1}^{+\infty} u_{k}\phi_{k}
\hbox{ such that }
\sum_{k=1}^{+\infty} \mu^{1/2}_{k}u^{2}_{k}<+\infty \right\};\qquad\qquad\quad
\end{align*}
\text{(ii)}$\quad  \hmezzo$  is an Hilbert space with the norm
$$
\|u\|_{\hmezzo}=\left\{\int_{\Omega}\int_{\Omega}
\frac{|u(x)-u(x')|^{2}}{|x-x'|^{N+1}}dxdx'\right\}^{1/2}
$$
equivalent to the usual one in $\sobmezzo$.
\epr
\bdim
Since $\Omega$ is of class $C^{2,\alpha}$, we have that $H^{1/2}(\Omega)$ can be equivalently characterized as $\left\{ u=\mathrm{tr}_\Omega v:\, v\in H^1(\mathC)\right\}$,
where we write tr$v= v|_\Omega = v (\cdot,0)$. Then Proposition
\ref{propo:mediacontinua} provides the inclusion
$$
\left\{ u=\mathrm{tr}_\Omega v:\, v\in\huno\right\} \subset \hmezzo.
$$
In order to show the opposite one, consider $u\in\hmezzo$ and consider
$v\in H^{1}(\mathC)$ such that $u=\mathrm{tr}_\Omega v$. Notice that
$Z_{\mathC}v\in \huno$ and Proposition \ref{propo:mediacontinua} implies  that
$$
\mathrm{tr}_\Omega (Z_{\mathC}v)=\mathrm{tr}_\Omega (v-A_{\mathC}v)
=u-\slashint_\Omega v(x,0)\,dx=u
$$
then we have found $\tilde{v}=Z_{\mathC}v$ belonging to $\huno$ and such
that  $u=\mathrm{tr}_\Omega (\tilde{v})$, yielding the
first equality in {\it (i)}.
As far as the second equality is concerned, we start by proving the inclusion
\[
\left\{ u=\mathrm{tr}_\Omega v:\, v\in\huno\right\} \subset
\left\{ u\in L^{2}(\Omega) :\, u=\sum_{k=1}^{+\infty} u_{k}\phi_{k}
\hbox{ such that }
\sum_{k=1}^{+\infty} \mu^{1/2}_{k}u^{2}_{k}
<+\infty \right\}.
\]
Indeed any $v\in\huno$ can be written as $v(x,y)=\sum_{k\geq 1} v_k(y) \phi_k(x)$,  with
$$
\|v\|^2_{\huno} = \int_0^{+\infty}\left( \sum_{k\geq 1} \mu_k v_k(y)^2 + v_k'(y)^2\right)\,dy
$$
then
$$
\sum_{k\geq 1} \mu_k v_k(y)^2 <+\infty \quad \text{a.e. in $ (0,+\infty).$}
$$
Let us fix $\bar{y}$  such that
$\dyle\sum_{k\geq 1} \mu_k^{1/2}v_k(\bar y)^2$ is finite and take
$u=\tr_\Omega v = \dyle\sum_{k\geq 1} v_k(0) \phi_k(x)$. We have
\[
\begin{split}
\|v\|_{\huno}^{2} \geq
\sum_{k\geq 1} \int_0^{\bar y}2\left| \mu_k^{1/2} v_k(y) v_k'(y)\right|\,dy
 \geq \left| \sum_{k\geq 1}\mu_k^{1/2} v_k(\bar y)^2 -
\sum_{k\geq 1}\mu_k^{1/2} v_k(0)^2 \right|,
\end{split}
\]
implying the desired inclusion. On the other hand, let
$\sum_{k\geq 1}\mu_k^{1/2} u_k^2<+\infty$, and let us define
\beq\label{armo}
v(x,y)=\sum_{k=1}^{+\infty} u_{k}\phi_{k}(x)e^{-\mu^{1/2}_k y}.
\eeq
It is a direct check to verify that $v\in \huno$ (see also Lemma 2.10 in
\cite{cata}), obtaining that all the equalities in {\it (i)} hold.

Let us now show conclusion (ii), starting with proving that there
exist constants $A,\,B$ such that
\beq\label{norme}
A\|u\|_{\sobmezzo}\leq \|u\|_{\hmezzo}\leq B\|u\|_{\sobmezzo}
\eeq
As
$$
\|u\|_{\sobmezzo}^{2}=\|u\|_{L^{2}}^{2}+\int_{\Omega}\int_{\Omega}
\frac{|u(x)-u(x')|^{2}}{|x-x'|^{N+1}}dxdx'.
$$
The right hand side inequality holds for $B=1$; in order to show
the left hand side inequality, let us argue by contradiction
and suppose that there exists a sequence $u_{n}\in \hmezzo$,
with $\|u_{n}\|_{L^{2}(\Omega)}=1$ and $\|u_{n}\|_{\hmezzo}\to 0$. Then $u_{n}$ is
uniformly bounded in $\sobmezzo$ and there exists $u\in \hmezzo$
such that $u_{n}$  converges to $u$ weakly in $\sobmezzo$ and
strongly in $L^{2}(\Omega)$ (notice that we do not know that the
quantity $\|\cdot\|_{\hmezzo}$ is a norm on $\hmezzo$). As a consequence, $\|u\|_{L^{2}(\Omega)}=1$ and
$$
\int_{\Omega}\int_{\Omega}\frac{|u(x)-u(x')|^{2}}{|x-x'|^{N+1}}dxdx'=0,
$$
which is an obvious contradiction.
As a byproduct of inequalities \eqref{norme} we obtain that
$\|\cdot\|_{\hmezzo}$ is a well defined norm and
since $\hmezzo$ is a closed subspace of $\sobmezzo$
with respect to the usual norm conclusion (ii) holds.
\edim
Carefully reading the proof of the second equality in (i) of
the previous proposition, one realizes that for any
$u\in \hmezzo$ we can construct a suitable extension $v\in \huno$
which is harmonic and that can be written in terms of a Fourier
expansion as shown in \eqref{armo}. In the next lemma
we provide a variational characterization of such extension.
\begin{lemma}\label{armonic_extension}
For every $u\in\hmezzo$ there exists an unique $v\in\huno$
achieving
\[
\inf\left\{ \|v\|_{\huno}^{2}=\int_\mathC |\nabla v(x,y)|^2\,dxdy
:\, v\in\huno,\,v(\cdot,0)=u\right\}.
\]
Moreover, the function $v$ is the unique (weak) solution of the problem
\beq\label{pbest}
\begin{cases}
\Delta v=0 &\text{ in } \mathC,
\\
\pa_{\nu} v=0 &\text{ on } \pa \Omega\times (0,+\infty),
\\
v(x,0)=u(x) &\text{ on } \Omega.
\end{cases}
\eeq
Finally,
\beq\label{est}
\hbox{if}\quad u(x)=\dyle\sum_{k=1}^{+\infty} u_{k}\phi_{k}(x)
\quad \text{then}\quad
v(x,y)=\sum_{k=1}^{+\infty} u_{k}\phi_{k}(x)e^{-\mu^{1/2}_k y}
\eeq
\end{lemma}
\begin{proof} We observe that the functional to be minimized is
simply the square of the norm in $\huno$, and the set on which we
minimize is non empty and weakly closed thanks to the compact
embedding of $H^{1/2}(\Omega)$ in $L^p(\Omega)$, for any
exponent $p<2N/(N-1)$. The strict convexity of the functional implies
the existence and uniqueness of the minimum point.

As usual, the unique minimum point $v$ satisfies the boundary condition
on $\Omega$ (in the $H^{1/2}$-sense) by constraint, and
$$
\int_{\mathC} \nabla v\cdot\nabla \psi=0
\qquad \forall\psi\in\huno \hbox{ s.t. } \psi(x,0)\equiv0 .
$$
As a consequence, for every $\zeta\in H^1(\mathC)$ such that $\zeta(x,0)\equiv0$, it is possible to choose
$\psi:=\zeta-A_\mathC\zeta$ as a test function in the previous equation. This provides
\begin{equation}\label{eq:var_H^1}
\begin{split}
0&=\int_{\mathC} \nabla v\cdot\nabla \zeta -\int_{\mathC} \partial_{y} v(x,y) ( A_\mathC\zeta)'(y)\,dxdy \\
&=\int_{\mathC} \nabla v\cdot\nabla \zeta -|\Omega|\int_{0}^{+\infty} (A_\mathC v)'(y)
(A_\mathC\zeta)'(y)\,dy\\
&=\int_{\mathC} \nabla v\cdot\nabla \zeta\qquad\qquad\qquad\forall\zeta\in H^1(\mathC)\hbox{ s.t. } \zeta(x,0)\equiv0.
\end{split}
\end{equation}
In a standard way this implies both that $v$ is harmonic in $\mathC$ and that it satisfies the
boundary condition on $\pa \Omega\times (0,+\infty)$ (in the $H^{-1/2}$-sense).

Finally, if $u(x)$ is given as in \eqref{est}, then $v$ as in
\eqref{est} solves problem \eqref{pbest} and the uniqueness of the solution
provides the claim.
\end{proof}
\begin{definition}\rm
We will refer to the unique $v$ solving \eqref{pbest} as the \emph{Neumann
harmonic extension} of the function $u$.
\end{definition}
\begin{remark}\label{rem:equiv_norms_in_Hmezzo}\rm
As we already noticed,
\[
\begin{split}
H^{1/2}(\Omega)
 &=\left\{ u\in L^2(\Omega):\,\|u\|_{L^{2}}^{2}+\int_{\Omega}\int_{\Omega}
   \frac{|u(x)-u(x')|^{2}}{|x-x'|^{N+1}}dxdx'<+\infty \right\}\\
 &=\left\{ u=\mathrm{tr}_\Omega v:\, v\in H^1(\mathC)\right\}.
\end{split}
\]
Furthermore, it is well known that the two norms
\[
\begin{split}
\|u\|_{\sobmezzo,1}^{2}&=\|u\|_{L^{2}}^{2}+\int_{\Omega}\int_{\Omega}
   \frac{|u(x)-u(x')|^{2}}{|x-x'|^{N+1}}dxdx',\\
\|u\|_{\sobmezzo,2}^{2}&=\inf\left\{ \|v\|_{H^1(\mathC)}^{2}:\,v\in H^1(\mathC),
  \,v(\cdot,0)=u\right\}
\end{split}
\]
are equivalent. Reasoning as in the proof of Proposition \ref{prop:H1/2}, and taking into account
Lemma \ref{armonic_extension}, we obtain that $\hmezzo$ can be equipped with the equivalent norms
\[
\begin{split}
\|u\|_{\hmezzo,1}^{2}&=\int_{\Omega}\int_{\Omega}
   \frac{|u(x)-u(x')|^{2}}{|x-x'|^{N+1}}dxdx',\\
\|u\|_{\hmezzo,2}^{2}&=\inf\left\{ \|v\|_{\huno}^{2}:\,v\in \huno,
  \,v(\cdot,0)=u\right\}=\sum_{k=1}^{+\infty} \mu^{1/2}_{k}u^{2}_{k},
\end{split}
\]
where the terms $u_k$ are the Fourier coefficients of $u$.
In particular, the harmonic extension of $u$ depends on $u$ in a linear
and continuous way.
\end{remark}
In order to introduce and study
the dual space of $\hmezzo$ let us first introduce the following space.
\begin{definition}\rm
Let us define the following subspace of $H^{-1/2}(\Omega)$.
$$
\acca^{-1/2}(\Omega):=\left\{f\in H^{-1/2}(\Omega) :\,
\langle f,1\rangle=0\right\},
$$
where $\langle \cdot,\cdot\rangle$ denotes the duality pairing.
\end{definition}
The subspace just introduced as a strict connection with the dual space of $\hmezzo$ as
well explained in the following proposition.

\begin{proposition}\label{prop:H-1/2}
It holds
\(
\hmezzo^* \cong \acca^{-1/2}(\Omega).
\)
\end{proposition}
\begin{proof}
We can exploit the splitting
$\sobmezzo = \hmezzo \oplus \R$ in order to obtain
\[
\hmezzo^* = H^{-1/2}(\Omega)/\sim\qquad\text{where }f_1\sim f_2 \iff
f_1|_{\hmezzo}=f_2|_{\hmezzo}.
\]
More precisely, on one hand if $g\in\hmezzo^*$ then, for every $c\in\R$,
\[
f := g\circ Z_\Omega + c A_\Omega \in H^{-1/2}(\Omega);
\]
on the other hand, if $f\in H^{-1/2}(\Omega)$ then $g:=f|_{\hmezzo} \in \hmezzo^*$ and
\[
f = g\circ Z_\Omega + \langle f,1\rangle A_\Omega.
\]
Moreover, both the maps defined above are linear and continuous. This proves
that $\hmezzo^*$ is isomorphic to $\left\{f\in H^{-1/2}(\Omega) :\, \langle f,1\rangle=c\right\}$,
for every fixed $c$, and in particular for $c=0$.
\end{proof}

As a first step to arrive to a correct definition of the half Laplacian
operator, let us prove  the following lemma
\begin{lemma}
Let $u\in \hmezzo$, and let $v\in\huno$ denote its Neumann harmonic extension. Then the functional
$\partial_\nu v\Big\rvert_{\Omega\times\{0\}}=-\partial_y v(\cdot,0)
\from \sobmezzo\to \R$
is well defined as
\[
\langle -\partial_y v(\cdot,0), g \rangle := \int_{\mathC} \nabla v \cdot \nabla \tilde g\,dxdy,
\]
where $g\in H^{1/2}(\Omega)$ and $\tilde g$ is any $H^1(\mathC)$-extension of $g$. Moreover,
\[
-\partial_y v(\cdot,0)\in\acca^{-1/2}(\Omega).
\]
\end{lemma}
\begin{proof}
The functional is well defined, indeed if $\tilde g_1$ and $\tilde g_2$ are two
extensions of $g$ we have that $(\tilde g_2-\tilde g_1)(x,0)\equiv0$ and,
arguing as in equation \eqref{eq:var_H^1}, yields
\[
\int_{\mathC} \nabla v \cdot \nabla \tilde g_2\,dxdy-\int_{\mathC}
\nabla v \cdot \nabla \tilde g_1\,dxdy=
\int_{\mathC} \nabla v \cdot \nabla (\tilde g_2-\tilde g_1)\,dxdy = 0.
\]
Moreover $-\partial_y v(x,0)$ is linear and continuous: indeed, let us choose
as an extension of $g$ $G:=A_\Omega g +\tilde g$, where $\tilde g$ is the
harmonic extension of $Z_\Omega g$; by Remark
\ref{rem:equiv_norms_in_Hmezzo} applied to $\tilde g$ we have that
\[
\begin{split}
|\langle -\partial_y v(\cdot,0), g \rangle |^2
 &\leq \| v \|^2_{\huno}\left(|A_\Omega g|^2+\|\tilde g\|^2_{\huno}\right)\\
 &\leq C\left(\|g\|_{L^{2}}^2 + \|g\|_{\hmezzo}^2\right)= C \|g\|_{H^{1/2}(\Omega)}^2.
\end{split}
\]
As a consequence $-\partial_y v(x,0)\in H^{-1/2}(\Omega)$.
Finally, since $w(x,y):= (1-y)^+$ belongs to $H^1(\mathC)$, by definition we obtain that
\[
\langle -\partial_y v(\cdot,0), 1 \rangle = \int_{\mathC} \nabla v \cdot \nabla w\,dxdy =
-\int_{\mathC\cap\{y<1\}}v_y\,dxdy =\int_\Omega [-v(x,1) +  v(x,0)]\,dx,
\]
which vanishes because $v\in\huno$.
\end{proof}

\begin{remark}\rm
If the harmonic extension $v$ is more regular (for instance $H^2(\mathC)$),
then we can employ integration by parts in order to prove that the definition
of $-\partial_y v(x,0)$ given above agrees with the usual one.
\end{remark}
Thanks to the previous lemmas, we are now in a position
to define the fractional operators we work with.
\begin{definition}\label{df:lmezzo}\rm
We define the operator $L_{1/2} \from \hmezzo\to \acca^{-1/2}(\Omega)$ as
\beq\label{defL}
L_{1/2} u = -\partial_y v(\cdot,0).
\eeq
where $v$ is the harmonic extension of $u$ according to \eqref{pbest}.
Analogously, we define the operator $\lapneu\from H^{1/2}(\Omega) \to H^{-1/2}(\Omega)$ by
\[
\lapneu = L_{1/2}\circ Z_\Omega.
\]
\end{definition}

In Definition \ref{df:lmezzo} we have introduced the fractional Laplace operator asso\-cia\-ted to
homogeneous Neumann boundary conditions as a Dirichlet to Neumann map. Moreover,
thanks to the equivalences of Proposition \ref{prop:H1/2}, we realize the spectral
expression of this operator as explained in the following remark.

\begin{remark}\label{rem:L}\rm
Since the harmonic extension operator $u\mapsto v$ is linear and continuous by Remark
\ref{rem:equiv_norms_in_Hmezzo}, we have that
both $L_{1/2}$ and $\lapneu$ are linear and continuous. Moreover, if $u\in H^1(\Omega)$ and
$u(x)=\dyle\sum_{k=1}^{+\infty} u_{k}\phi_{k}(x)$, we can use equation \eqref{est} to infer that
$\partial_y v(x,0) \in L^2(\Omega)$. This allows to write
\[
\lapneu u(x) = L_{1/2}(u)(x)=-\pa_{y}v(x,0)
=\sum_{k=1}^{+\infty} \mu^{1/2}_{k} u_{k} \phi_{k}(x).
\]
In particular, if $u\in H^2(\Omega)$ then
\[
\lapneu\circ\lapneu u=-\Delta_{\mathcal N} u
\]
provides the usual Laplace operator associated to homogeneous Neumann
boun\-da\-ry conditions on $\partial \Omega$.
\end{remark}
We remark that we can think to $L_{1/2}$ as acting between $\hmezzo$ and its dual thank to
Proposition \ref{prop:H-1/2}.
While $\lapneu\from H^{1/2}(\Omega) \to H^{-1/2}(\Omega)$ is neither injective nor surjective,
we have that $L_{1/2} \from \hmezzo\to \acca^{-1/2}(\Omega)$ is invertible.
\begin{lemma}\label{lem:weak-v}
For every $f\in \acca^{-1/2}(\Omega)$ there exists a unique $v\in\huno$ such that
\beq\label{equalin}
\intc\nabla v(x,y)\cdot\nabla\psi (x,y)dxdy=\langle f,\psi(\cdot,0)\rangle
\qquad \forall\,\psi\in \huno.
\eeq
Moreover, the function $v$ is the unique (weak) solution of the problem
\beq\label{linv}
\begin{cases}
\Delta v=0 & \hbox{in }\mathC, \\
\pa_{\nu} v=0 & \hbox{on } \pa\Omega\times(0,+\infty),\\
\pa_{\nu} v(x,0)= f(x) & \hbox{on } \Omega.
\end{cases}
\eeq
\end{lemma}
\begin{proof}
The existence and uniqueness of $v$ follows from Riesz representation Theorem. The fact that
$v$ satisfies \eqref{linv} follows once one shows that equation \eqref{equalin} holds
also for every $\psi\in H^1(\mathC)$. This can be readily done exactly as in the proof of
Lemma \ref{armonic_extension}. Further, it is also a consequence of the next result.
\end{proof}
The choice of $\huno$ as test function space is not restrictive
as the following lemma shows.
\begin{lemma}\label{test}
Let $f\in \acca^{-1/2}(\Omega)$, and $v\in \huno$ be defined as in Lemma \ref{lem:weak-v}.
Then:
\begin{itemize}
\item[(i)] there exist positive constants $C$, $k$ depending on $f$ such that, for every $y>0$,
\[
\int_\Omega |\nabla v(x,y)|^2\,dx \leq C e^{-ky};
\]
\item[(ii)] equation \eqref{equalin} holds for any $\psi\in H^1_{\mathrm{loc}}(\overline{\mathC})$
admitting a constant $C'$ such that, for every $y$,
\beq\label{testequa}
\|\psi(\cdot,y)\|_{L^2(\Omega)}\leq C'.
\eeq
\end{itemize}
\end{lemma}
\begin{proof}
Let $\psi\in H^1(\Omega\times(0,y))$. Since $\Omega\times(0,y)$ is bounded,
we can test \eqref{linv} with $\psi$ and
use integration by parts in order to obtain that, for a.e. $y$,
\begin{equation}\label{eq:byparts}
\int_{\mathC\cap\{t<y\}} \nabla v(x,t)\cdot\nabla \psi (x,t)\,dxdt
=\langle f, \psi(\cdot,0)\rangle + \into v_y(x,y) \psi(x,y)\,dx.
\end{equation}
As far as the first statement is concerned, the above equation used with $\psi=v$
gives
\[
\begin{split}
\Phi(y)&:=\int_y^{+\infty} \int_\Omega |\nabla v(x,t)|^2\,dxdt \\
&=  \int_\mathC |\nabla v(x,t)|^2\,dxdt - \int_{\mathC\cap\{t<y\}}  |\nabla v(x,t)|^2\,dxdt\\
&=  \langle f, v(\cdot,0)\rangle - \langle f, v(\cdot,0)\rangle - \into v_y(x,y) v(x,y)\,dx.
\end{split}
\]
Then $\Phi$ is absolutely continuous and
\[
\begin{split}
\Phi(y) &= -\int_\Omega v(x,y)v_y(x,y)\,dx\leq \left(\int_\Omega v^2(x,y)\,dx\right)^{1/2}
           \left(\int_\Omega v_y^2(x,y)\,dx\right)^{1/2}\\
  &\leq c_{pw}\left(\int_\Omega |\nabla_x v(x,y)|^2\,dx\right)^{1/2}
           \left(\int_\Omega v_y^2(x,y)\,dx\right)^{1/2}\\
  &\leq\frac{c_{pw}}{2}\int_\Omega |\nabla v(x,y)|^2\,dx = -k\Phi'(y),
\end{split}
\]
which implies $\Phi(y)\leq \Phi(0)\cdot e^{-k y}$ and the required inequality.

Now we turn to the second statement. If $\psi$ is as in its assumption then \eqref{eq:byparts}
holds. In order to conclude we must prove that the last term in that
equation vanishes as $y\to+\infty$. But this is easily proved by applying
H\"older inequality and using the first part of the lemma.
\end{proof}
\begin{remark}\rm
In particular, the previous proposition implies that equation \eqref{equalin}
holds for any $\psi\in\sob$. On the other hand, from its proof one can deduce
that more general test functions are admissible, for instance functions such
that their $L^2(\Omega)$ norm does not grow too much with respect to $y$.
\end{remark}
We are now in the position to define the inverse operator of $L_{1/2}$.
\bdf\label{defiweak}\rm
We define the operator
$T_{1/2}\from\acca^{-1/2}(\Omega)\to\hmezzo$ by
\beq\label{deft}
T_{1/2}(f)=tr_{\Omega}(v)=v(x,0)
\eeq
where $v$ solves \eqref{linv}.
\edf

We collect in the  following proposition the properties of $T_{1/2}$.
\bpr\label{inverse}
The operator $T_{1/2}$ defined in \eqref{deft} is linear and
such that $L_{1/2}\circ T_{1/2}=T_{1/2}\circ L_{1/2}=Id$.

Moreover $T_{1/2}\from\elle{2}:=\big\{f\in L^{2}(\Omega), \, :\,
\intm_{\Omega}f(x)dx=0\big\}\to \elle{2}$ is compact, positive,
self-adjoint and $T_{1/2}\circ T_{1/2}=(-\Delta_{{\mathcal N}})^{-1}|_{\elle{2}}$.
\epr
\begin{proof}
First, let us observe that $T_{1/2}$ is well defined, as for every
$f\in \acca^{-1/2}(\Omega)$ there exists a
unique $v$ solution of \eqref{linv}, moreover
$T_{1/2}$ is evidently linear.
If $v(x,0)=T_{1/2}(f)$, where $v$ is the solution of \eqref{linv},
then $L_{1/2}v=\partial_{\nu}v(x,0)$ and from \eqref{linv},
$L_{1/2}T_{1/2}(f)=\partial_{\nu}v(x,0)=f(x)$, i.e. $T_{1/2}$ is the
inverse of the operator $L_{1/2}$.
\\
In order to show that $T_{1/2}$ is compact when restricted to $\elle{2}$,
let us take $f_{n}\in \elle{2}$ weakly converging to
$f\in \elle{2}$ and consider $T_{1/2}(f_{n})=v_{n}(x,0)$ with
$v_{n}\in\huno$ sequence of  solutions of \eqref{linv} with datum
$f_{n}$. From the weak formulation
of \eqref{linv} we obtain that $v_{n}$ is uniformly bounded
in $\huno$,
so that it weakly converges to a function $v\in \huno$, which turns
out to be a weak solution with datum $f$. Choosing as test function
$\psi=v_{n}-v$ in the equation satisfied by $v_{n}$
and taking advantage of the compact embedding of $\hmezzo$ in
$\elle{2}$ immediately gives the strong convergence
of $v_{n}$ to $v$ in $\huno$. And by
continuity of the trace operator, $v_{n}(x,0)=T_{1/2}(f_{n})$
converges to $v(x,0)=T_{1/2}(f)$ in $\elle{2}$.
\\
Arguing as in Proposition 2.12 in \cite{cata} it is easy to obtain that
$T_{1/2}$ restricted to $\elle{2}$ is self-adjoint and positive.
\\
Finally, the last part of the statement can be proved by following the argument
of Proposition 2.12 in \cite{cata} (see also Remark \ref{rem:L}).
\end{proof}

To end this section, we face some regularity issues. As already observed,
any of the above harmonic extensions is of course smooth inside $\mathC$. On the
other hand, improved regularity up to the boundary seems to be prevented by the
fact that $\partial\mathC$ is only Lipschitz. Nonetheless, we can exploit the homogeneous Neumann
condition (together with some regularity of $\partial\Omega$) in order to suitably extend the
harmonic functions outside $\mathC$, thus removing that obstruction.
\begin{proposition}\label{prop:reg_inv}
Let $\Omega$ be of class $C^{2,\alpha}$, and let
$f\in \acca^{-1/2}(\Omega)$, $v\in\mathcal{H}^1(\mathC)$ satisfy \eqref{equalin}. Then
$v\in C^{2,\alpha}(\overline{\Omega}\times(0,+\infty))$ and
\begin{itemize}
 \item[(i)] if $f\in L^{p}(\Omega)$, $2\leq p<\infty$, then $v\in W^{1,p}(\mathC)$ and
             $\|v\|_{W^{1,p}}\leq C(\Omega,p)\|f\|_{L^p}$;
 \item[(ii)] if $f\in W^{1,p}(\Omega)$, $2\leq p<\infty$, then $v\in W^{2,p}(\mathC)$ and
             $\|v\|_{W^{2,p}}\leq C(\Omega,p)\|f\|_{W^{1,p}}$;
 \item[(iii)] if $f\in C^{0,\alpha}(\overline{\Omega})$ then $v\in C^{1,\alpha}(\overline{\mathC})$ and
             $\|v\|_{C^{1,\alpha}}\leq C(\Omega,\alpha)\|f\|_{C^{0,\alpha}}$.
\end{itemize}
\end{proposition}
\begin{proof}
The fact that $v\in C^{2,\alpha}(\overline{\Omega}\times(0,+\infty))$ follows from standard
regularity theory for the Laplace equation with homogeneous Neumann boundary conditions
on smooth domains. As far as (i) is concerned, due to the exponential decay of $v$ given by
Lemma \ref{test}, we are left to prove regularity near $\{y=0\}$.
To start with, for any $x_{0}\in\Omega$, let us consider any half-ball
$$
B^+=B^+_{R}(x_0,0) = \left\{(x,y):\,|(x-x_0,y)|<R,\,y>0\right\}\subset \mathC
$$
and let us introduce the notation
\[
H^{1;0}_0(B^+):=\left\{\psi\in H^1(B^+):\,\psi|_{\partial B^+\cap \{y>0\}}\equiv 0\right\},\quad
a(v,\psi):=\int_{B^+} \nabla v\cdot\nabla\psi \,dxdy.
\]

Since $v$ solves \eqref{equalin}, integration by parts yields
\begin{align}\label{eq:reg_easy}
a(v,\psi) & = \int_{B^{+}\cap\{y=0)\}}f(x)\psi(x,0)\,dx = \int_{\partial B^+} f(x)\psi(x,y)\,d\sigma \\
 & = -\int_{B^+} f(x)\partial_y \psi(x,y)\,dxdy,
\quad \forall\psi\in H^{1;0}_0(B^+). \nonumber
\end{align}
As a consequence, Theorem 3.14 in \cite{trobook}
implies that $v\in W^{1,p}(B^+_r)(x_0,0)$ for e\-ve\-ry $r<R$.
On the other hand, let $x_0\in \partial\Omega$. By assumption, there
exists an open neighborhood $U\ni(x_0,0)$ and a
$C^{2,\alpha}$-diffeomorphism $\Phi$ between $U\cap\{y>0\}$ and $B^+_1(0,0)$
which is the identity on the $y$-coordinate and such that $\Phi(x_0,0)=(0,0)$,
$\Phi(U\cap\mathC)=B^+\cap\{x_N<0\}$.
Let $\tilde  v = v \circ \Phi^{-1}$. Since $v$ is harmonic we have that $\tilde v$
satisfies an equation like \eqref{eq:reg_easy} on $B^+\cap\{x_N<0\}$, where
now the bilinear form $a$ has $C^{1,\alpha}$ coefficients
(which depend on $\Omega$ through the first derivatives of $\Phi$). Accordingly,
the conormal derivative of $\tilde v$ on $\{x_N=0\}$ vanishes.
Since $\Omega$ is $C^{2,\alpha}$, the last fact allows
to extend $\tilde v$ to the whole $B^+$ by (conormal) reflection, at least when
the initial neighborhood $U$ is sufficiently small; in a standard way, the extended
function satisfies again an equation like \eqref{eq:reg_easy}, and now the
corresponding $a$ has Lipschitz-continuous coefficients.
Furthermore, the analogous extension of $f$ is again $L^p$. As a consequence,
Theorem 3.14 in \cite{trobook} implies also in this situation that $\tilde v$, and hence $v$,
is $W^{1,p}(B^+_r)$ for $r<R$.
Taking into account the previous discussion, property (i) follows by a covering
argument. Finally, (ii) and (iii) can be proved with minor changes in the previous
argument, by using Theorems 3.15, 3.12 and 1.17 in \cite{trobook}.
\end{proof}

The previous proposition implies a number of regularity properties for
the inverse operator $T_{1/2}$. Analogous arguments yield improved
regularity also for the direct operator $L_{1/2}$.
\begin{proposition}\label{regarmo}
Let $\Omega$ be of class $C^{2,\alpha}$, and let $u\in \hmezzo$ be such that
\[
u\in C^{1,\alpha}(\overline{\Omega})\quad\text{ and }
\quad\partial_\nu u = 0 \text{ on }\partial\Omega.
\]
Finally, let $v\in\mathcal{H}^1(\mathC)$ be the Neumann harmonic extension
of $u$ according to Lemma \ref{armonic_extension}. Then
$v\in C^{1,\alpha}(\overline{\mathC})$, and
$\|v\|_{C^{1,\alpha}(\mathC)}\leq C(\Omega,\alpha)\|u\|_{C^{1,\alpha}(\Omega)}$.
\end{proposition}
\begin{proof}
It is sufficient to show that $w(x,y):=v(x,y)-u(x)$ is $C^{1,\alpha}(\overline{\mathC})$.
This can be done by following straightforwardly the proof of Proposition \ref{prop:reg_inv},
once one notices that, instead of equation \eqref{eq:reg_easy}, $w$ satisfies $w(x,0)=0$
and, as $v$ is harmonic,
\[
a( w,\psi) = \int_{B^+} -\nabla_x u(x)\cdot \nabla_x\psi(x,y)\,dxdy\quad
\forall\psi\in H^{1}_0(B^+),
\]
for every $B^+\subset\mathC$. Hence the role that $f$ had in the aforementioned
proposition is now played by $\nabla_x u$. Since $\nabla_x u$ is
$C^{0,\alpha}(\overline{\mathC})$, the proposition follows again by applying
\cite{trobook}, Theorem 3.12, to $w$ (or to suitable extensions
$\tilde w$ near $\partial\Omega\times\{0\}$).
\end{proof}

As a conclusion of this section we state the following result, which
will be useful in the applications

\begin{corollary}\label{coro:smooth_L}
Let us define the spaces
$$
X  :=\left\{ u\in  C^{1,\alpha}(\ov\Omega):\, A_{\Omega} u =0,\,
\partial_{\nu}u(x)=0\; \text{on}\;\, \partial \Omega\right\},\;
F:=\left\{ f\in  C^{0,\alpha}(\ov\Omega):\, A_{\Omega}f =0\right\}.
$$
Then the operators
\[
L_{1/2}\from X \to F,\qquad T_{1/2}\from F \to X
\]
are linear and continuous and $L_{1/2}\circ T_{1/2}=T_{1/2}\circ L_{1/2}=Id$.
\end{corollary}

\begin{proof}
The conclusion easily follows from Propositions \ref{inverse} and \ref{prop:reg_inv}.
\end{proof}
In the following we will be concerned with positive solutions of equations involving
the fractional operators defined above. In this perspective, the arguments we employed
to improve regularity allow to check the validity of suitable maximum principles and Hopf lemma.
In particular, the following strong maximum principle holds.
\bpr\label{maxi}
Let $c\in L^{\infty}(\Omega)$ and nonnegative. Every
$u\in C^{1,\alpha}(\overline{\Omega})$ satisfying
\beq\label{prmax}
\begin{cases}
\lapneu u + c(x)u\geq0 &  \text{ in } \Omega \\
u\geq 0  & \text{ in } \Omega,
\end{cases}
\eeq
is either identically zero or strictly positive on $\overline{\Omega}$.
\epr
\begin{proof}
Let us write $u=Z_\Omega u + A_\Omega u =: \tilde u + c_u$, and let $\tilde v$
denote the Neumann harmonic extension of $\tilde u$ to $\mathC$.
Then $v(x,y):=\tilde v(x,y) + c_u$ is harmonic, non-negative, and $v(x,0)=u(x)$.
Now, if $u(x_0)=0$ for some $x_0\in\Omega$, this would imply
$\partial_\nu v(x_0,0) = (-\Delta)^{1/2} u (x_0)\geq 0$, in contradiction with the
Hopf principle for harmonic functions. On the other hand, if $u(x_0)=0$ for
some $x_0\in\partial\Omega$, we can argue in the same way, considering instead of $v$
its conormal even extension, as in the proof of Proposition \ref{prop:reg_inv}.
\end{proof}

\bos\label{maxiv}
As a direct consequence of Proposition \ref{maxi} and of classical maximum principle
for harmonic functions, we deduce that, if $u>0$ satisfies \eqref{prmax}, then its harmonic
extension $v=\tilde v+c_u$ is positive in $\overline{\mathC}$.
\eos

\section{The Weighted Logistic Equation}\label{sec:logistica}
Our main application is the study of the positive solutions
of the nonlinear Problem \eqref{logim}, understood in terms of the operator
$\lapneu$.
To this aim, a necessary solvability condition is that the right hand side
of the equation has null average. On the other hand the possible solution $u$,
being positive, has positive average. In order to apply the theory developed in
the previous section, we recall that any $u\in H^{1/2}(\Omega)$ can be decomposed
as
\begin{equation}\label{eq:decomptilde}
u = A_\Omega u + Z_\Omega u =: \tilde u + c_u,
\end{equation}
where $c_u$ is constant and $\tilde u\in \hmezzo$. Using Lemma \ref{armonic_extension}
we can denote by $\tilde v$ the Neumann harmonic extension of $\tilde u$ to $\mathC$,
obtaining that
\[
v(x,y):= \tilde v(x,y) + c_u
\]
is harmonic and $v(x,0)=u(x)$. It is worthwhile noticing that, as far as $c_u\neq0$,
$v\not\in H^1(\mathC)$.

Taking into account the previous discussion, we can now define what we 
mean by a weak solution of a general nonlinear problem.
\bdf\label{nonlinear}\rm
A weak solution $u\from\Omega\to\R$ of the nonlinear problem
\beq\label{nonlinequa}
\begin{cases}
(-\Delta)^{1/2} u= f(x,u) & \hbox{in } \Omega, \\
\pa_{\nu} u=0 & \hbox{on } \pa\Omega,
\end{cases}
\eeq
is a function  $u\in H^{1/2}(\Omega)$ such that
$f(\cdot,u(\cdot))\in H^{-1/2}(\Omega)$ and
\[
\text{both }\quad A_\Omega f(\cdot, u) = 0
\qquad\text{and }\quad\lapneu u = f(\cdot,u).
\]
In particular $u(x)=v(x,0)$, where $v(x,y)=\tiv(x,y)+h$, and
$(\tiv,h)\in \huno\times \R$ is a
weak solution of the nonlinear problem
\beq\label{nonlinprob}
\begin{cases}
\Delta \tiv=0 & \hbox{in }\mathC,
\\
\pa_{\nu} \tiv=0 & \hbox{on } \pa\Omega\times(0,+\infty), \\
\pa_{\nu} \tiv= f(x,\tiv+h) & \hbox{on } \Omega\times\{0\},
\\
\dys \into f(x,\tiv(x,0)+h)dx=0,
\end{cases}
\eeq
in the sense that
\beq\label{nonlinweak}
\begin{cases}
\dys\int_{\mathC}\nabla \tiv\nabla \psi\,dxdy= \into f(x,\tiv+h)\psi\,dx
\qquad \forall\,\psi\in \huno,
\medskip\\
\dys\into f(x,\tiv(x,0)+h)dx=0.
\end{cases}
\eeq
\edf

Using the previous definition we can now rewrite Problem \eqref{logim} in the equivalent form
\begin{equation}\label{logim2}
\begin{cases}
\dys \lapneu u = \lambda u(m(x)-u),\\
\dys \lambda\int_\Omega u(m(x)-u)\,dx=0,
\end{cases}
\end{equation}
and we recall that we assume
\[
\lambda>0\qquad\text{ and }\qquad m\in C^{0,1}(\overline{\Omega}).
\]
\bos
In the standard diffusion case, nonlinear boundary data have been
frequently considered especially in the determination of
selection-migration problem for alleles in a region,
admitting flow of genes throughout the boundary (see
\cite{mana} and the references therein).
\eos
As in the classical literature concerning the logistic equation,
the comprehension of the linearized problem arises as crucial in the study.
In our context, this correspond to tackle the following weighted eigenvalue problem
\beq\label{eigen}
\begin{cases}
\dys \lapneu u = \lambda m(x)u,\\
\dys \lambda\int_\Omega m(x)u\,dx=0.
\end{cases}
\eeq
\begin{remark}\rm
When $m\equiv1$, the nontrivial solutions of
$$
\begin{cases}
\dys \lapneu\vfi =\la \vfi,
\\
\dys \int_\Omega\vfi\,dx=0,
\end{cases}
$$
are $\vfi_{k}=\phi_{k}$ associated to $\la_{k}=\sqrt{\mu_{k}}$
where $\phi_{k}$ and $\mu_{k}$ are respectively eigenfunctions
and eigenvalues of the usual Laplace operator $-\Delta$ with homogeneous
Neumann boundary conditions as in \eqref{autof}.
\end{remark}
\begin{remark}\label{rem:fredholm}\rm
Taking into account the usual decomposition $u=\tilde u+c_u$ as in \eqref{eq:decomptilde},
we have that Problem \eqref{eigen} can be written as
\[
\begin{cases}
\dys \tilde u - \lambda T_{1/2}(m(\tilde u + c_u))=0,\\
\dys \int_\Omega m(\tilde u + c_u)\,dx=0.
\end{cases}
\]
where $T_{1/2}$ is compact by Proposition \ref{inverse}. If moreover we assume
$\int_\Omega m  \neq 0$, we can solve the second equation for $c_u$ and infer
the equivalent formulation
\[
\tilde u - \lambda T_{1/2}\left(m\tilde u - \frac{\int_\Omega m\tilde u}{\int_\Omega m}m\right)=0.
\]
As a consequence, we can apply Fredholm's Alternative, obtaining that the spectrum of the
operator at the left hand side consists in a sequence of eigenvalues $(\lambda_k)_k$, with
associated kernel of dimension $d_k<+\infty$ and  closed range having codimension $d_k$.
\end{remark}
\begin{lemma}\label{infty}
Any nontrivial solution $u$ of Problem \eqref{eigen} is of class
$C^{1,\alpha}(\overline{\Omega})$. Moreover, $u\geq0$ implies $u>0$ in $\overline{\Omega}$.
\end{lemma}
\begin{proof}
The proof relies on the classical bootstrap technique. Indeed, as above, let us write
$u=\tilde u+c_u$ and let us denote with $\tilde v$ the Neumann harmonic extension of $\tilde u$.
Since $m\in C^{0,1}(\overline{\Omega})$,
Proposition \ref{prop:reg_inv} and the trace and Sobolev embedding theorems imply
\begin{multline*}
\tilde v \in W^{1,r}(\mathC) \implies \tilde u \in W^{1-1/r,r}(\Omega) \implies
\tilde u\in L^{Nr/(N+1-r)}(\Omega) \\
\implies \lambda mu \in L^{Nr/(N+1-r)}(\Omega)  \implies
\tilde v \in W^{1,Nr/(N+1-r)}(\mathC),
\end{multline*}
whenever $2\leq r < N+1$. Starting from $r_0=2$ and iterating the above procedure the
first part of the proposition follows. As a consequence, the second one is implied by
Proposition \ref{maxi}.
\end{proof}
Searching for positive solutions of \eqref{logim2}, we are interested in positive eigenfunctions
of \eqref{eigen}. Of course, $\la=0$ is always eigenvalue with normalized eigenfunction
$\vfi_{0}=1/\sqrt{|\Omega|}>0$, but this does not prevent the existence of positive eigenfunctions
associated with positive eigenvalues.
\begin{lemma}\label{infty2}
If there exists a positive eigenvalue $\la_{1}$ with a positive
eigenfunction $\vfi_{1}$ then the function $m$ is such that
\beq
\label{ipomneg}
\into m(x)dx<0\quad\text{ and }\quad \exists\, x_{0}\in
\Omega\text{ such that }m(x_{0})>0.
\eeq
\end{lemma}
\bdim
Supposing that there exists $\la_{1}>0$  with positive nonconstant
eigenfunction $\vfi_{1}$,
we can apply Lemma \ref{test} and use $\vfi_{0}=1/\sqrt{|\Omega|}>0$ as a test function in the
weak formulation of \eqref{eigen} satisfied by
$\vfi_{1}$ to obtain
$$
0=\la_{1}\into m(x)\vfi_{1}(x)\,dx.
$$
As $\vfi_{1}$ is positive, $m$ has to change sign. Now,
taking advantage of the usual decomposition,
let us write $\vfi_{1}(x)=v_{1}(x,0)=\tiv_{1}(x,0)+h$. From Remark \ref{maxiv},
we deduce that $v_{1}>0$ on $\overline{\mathC}$, and Lemma \ref{test} allows to use
$\psi=1/v_{1}$ as test function in the equation satisfied by $\vfi_{1}$. We obtain
\[
-\int_{\mathC} \left|\frac{\nabla v_1}{v_1}\right|^{2}\,dxdy = \lambda_1 \int_\Omega m(x)dx,
\]
and the lemma follows.
\edim
The following result shows that the previous necessary condition is also  sufficient
in order to obtain the existence of a first positive eigenvalue with positive eigenfunction.
%
\bte \label{primoautopos}
Let us suppose that $m\in C^{0,1}(\overline{\Omega})$
satisfies condition \eqref{ipomneg}.
Then
\begin{itemize}
 \item[(i)] there exists $\la_{1}>0$ and $\vfi_1\in C^{1,\alpha}(\overline{\Omega})$,
 $\vfi_{1}>0$ in $\overline{\Omega}$, solution of \eqref{eigen} with $\la=\la_{1}$;
 \item[(ii)] $\la_{1}$  is  the smallest positive eigenvalue, it is simple and any other solution
of \eqref{eigen} with $\la>\la_{1}$ changes sign.
\end{itemize}
\ete
\bdim
We will find $\vfi_{1}$ solving the extension Problem
\eqref{nonlinprob}, via a  constrained minimization.
Namely, we look for a minimum of the functional
$E\from\huno\times\R\to\R$ defined by
$$
E(\tilde v,h)=\frac12\intc |\nabla \tiv|^{2}dxdy
$$
constrained on the manifold
\beq\label{defM}
{\mathcal M}={\big\{}(\tiv,h)\in \huno\times \R\,:\,
\into m(x)(\tiv(x,0)+h)^{2}dx=1{\big \}}.
\eeq
First, let us observe that ${\mathcal M}\neq \emptyset$, indeed,
from \eqref{ipomneg} we can find $\omega$ an open set of positive
measure such that $m(x)>0$ in $\omega$, so that
for any $\tilde w\in C^{\infty}_{c}(\omega)$ having zero mean there exists a
suitable positive real number $C$, such that $(w/\sqrt{C},0)$
belongs to ${\mathcal M}$.
Let us consider a sequence $(\tiv_{n},h_{n})\in {\mathcal M}$
such that
$$
E(\tiv_{n},h_{n})\to \inf_{{\mathcal M}}E\geq 0.
$$
From the definition of $E$ it follows that $\tiv_{n}$ is uniformly bounded
in $\huno$, so that $\tiv_{n}(x,0)$ is uniformly bounded in
$H^{1/2}(\Omega)$; by the compact embedding of the trace
space $H^{1/2}(\Omega)$ in $L^{2}(\Omega)$ we obtain that
there exists $\tiv_{1}\in \huno$ such that $\tiv_{n}$ tends to $\tiv_{1}$ weakly
in $\huno$, and $\tiv_{n}(x,0)$ strongly converges to $\tiv_{1}(x,0)$ in
$L^{2}(\Omega)$. As far as the sequence $h_{n}$ is concerned,
let us show that it is bounded by contradiction, assuming that,
up to a subsequence, $h_{n}\to+\infty$ (the case $h_{n}\to -\infty$
can be handled analogously). By the definition of ${\mathcal M}$
it follows
$$
\lim_{n\to +\infty}h_{n}\into m(x)dx=-2\into m(x)\tilde v_n(x,0)dx
+\text{\small o}(1),
$$
where {\small o}$(1)$ denotes a quantity tending to zero as
$n$ goes to infinity.
Then, a contradiction follows from \eqref{ipomneg}.
As a consequence, there exists $h_{1}$ such that
$h_{n}\to h_{1}$ in $\R$. By weak lower
semicontinuity of $E$, it results that the pair $(\tiv_{1},h_{1})$ satisfies
\beq\label{inf}
E(\tiv_{1},h_{1})=\inf_{{\mathcal M}}E,\qquad \into m(x)(\tiv_{1}(x,0)+h_{1})^{2}dx=1.
\eeq
Moreover, let us show $E(\tiv_{1},h_{1})>0$. Again by contradiction, let us assume
that $E(\tiv_{1},h_{1})=0$; then, as $\tiv_{1}\in \huno$, it follows $\tiv_{1}=0$, and from
\eqref{inf} we obtain that
$$
h_{1}^{2}\into m(x)dx=1,
$$
so that \eqref{ipomneg} yields again a contradiction.
Since $(\tiv_{1},h_{1})$ is a constrained minimum point of $E$ on ${\mathcal M}$,
there exists $\la\in \R$ such that, by Proposition \ref{test}, for every
$\psi\in H^{1}_{\rm loc}(\ov{\mathC})$ satisfying \eqref{testequa}, it holds
\begin{align}\label{eqauto1}
\intc \nabla \tiv_{1}(x,y)\nabla \psi(x,y)\, dxdy =\la \into & m(x)(\tiv_{1}(x,0)+h_{1})\psi(x)dx,
\\
\label{eqauto2}\into m(x)(\tiv_{1}(x,0)+h_{1})dx &=0.
\end{align}

Choosing $\psi=\tiv_{1}+h_{1}$ we infer
$$
\inf_{{\mathcal M}}E=E(\tiv_{1}, h_1)=\intc |\nabla \tiv_{1}|^{2}\, dxdy =
\la \into m(x)(\tiv_{1}(x,0)+h_{1})^{2}dx=\la,
$$
thus $\la>0$ and  we can define
\beq\label{defla1}
\la_{1}:=\la=\dys\inf_{{\mathcal M}}E=E(\tiv,h_1)\quad\text{ and }\quad
\vfi_{1} (x) :=v_{1}(x,0)=\tiv_{1}(x,0)+h_{1}
\eeq
its corresponding eigenfunction, which is a weak solution of Problem \eqref{eigen}.
As a consequence, Lemma \ref{infty} implies that
$\vfi_1\in C^{1,\alpha}(\overline{\Omega})$ for every $\alpha\in(0,1)$.
Since any other solution $(\la,v)$ with
$\la>0$  of \eqref{eigen}
corresponds to a constrained critical point of $E$ on ${\mathcal M}$,
$\la_{1}$ is the smallest positive eigenvalue. In order to show
that  $\vfi_{1}$ can be chosen positive, let us
take  $w(x)=|\vfi_{1}(x)|=|\tiv_{1}(x,0)+h|$. Writing
$w(x)=\tiw(x)+c_{w}$, with $\tiw\in \hmezzo$ and $c_{w}$ constant,
let us consider $\tilde{\zeta}(x,y)\in \huno$ the harmonic extension of
$\tiw(x)$ obtained thanks to Lemma \ref{armonic_extension}.
Notice that $(\tilde{\zeta},c_{w})\in {\mathcal M}$; moreover
$\tilde{\zeta}(x,0)=|\tiv_1(x,0)+h|-c_{w}=z(x,0)$ for $z(x,y)=|\tiv_1(x,y)+h|-c_{w}$
so that, thanks to Remark \ref{rem:equiv_norms_in_Hmezzo}
$$
E(\tiw)\leq E(z)=\intc|\nabla|\tiv_1(x,y)+h||^{2}\leq \intc |\nabla \tiv_1|^{2}=E(\tiv_1).
$$
As a consequence, also the nonnegative function
$w$ solves the minimization
Problem \eqref{inf}, showing that we can assume,
without loss of generality, that $\vfi_1$ is nonnegative. But then Lemma \ref{infty}
applies again, yielding $\vfi_1>0$ on $\overline{\Omega}$.
It is possible to show that $\la_{1}$ is simple
by contradiction, supposing that there exists $\vfi_{1}$ and $u$
solutions of \eqref{defla1},
with $\vfi_{1}(x)=v_{1}(x,0)=\tiv_{1}(x,0)+h$,
$u(x)=w(x,0)=\tiw(x,0)+k$. From Remark \ref{maxiv},
we deduce that $v_{1}(x,y)>0$ in $\mathC$,
so that we can use $\psi(x,y)=w^{2}(x,y)/v_{1}(x,y)$ as
test function in the equation satisfied by $\vfi_{1}$, obtaining
$$
\la_{1}\into\!\!\!\! m(x)u^{2}(x)dx=\intc\!\! \nabla v_{1}(x,y)
\left[2\frac{w(x,y)}{v_{1}(x,y)}\nabla w(x,y)
-\left(\frac{w(x,y)}{v_{1}(x,y)}\right)^{2}\!\!\!\!\nabla v_{1}(x,y)
\right]\!dxdy.
$$
This implies
\begin{align*}
\la_{1}\into m(x)u^{2}(x)dx &=-\intc {\Big |}\nabla w(x,y)
-\frac{w(x,y)}{v_{1}(x,y)}\nabla v_1(x,y){\Big |}^{2}dxdy
\\
&+\la_{1}\into m(x)u^{2}(x)dx,
\end{align*}
that is
$$
0=-\intc {\Big |}\nabla w(x,y)-\frac{w(x,y)}{v_{1}(x,y)}
\nabla v_1(x,y)
{\Big |}^{2}dxdy=\intc v_{1}^{2}(x,y){\Big |}\nabla 
\left(\frac{w(x,y)}{v_{1}(x,y)}\right){\Big |}^{2}dxdy,
$$
yielding the linear dependence between $\vfi_{1}$ and $u$.
Moreover, it is possible to follow the same argument as in \cite{mi} to obtain that also the
algebraic multiplicity of $\la_{1}$ is one.
\\
Now we come to part (ii). In order to show that there is not a positive solution
$u$ of Problem \eqref{eigen} associated to
$\la>\la_{1}$,  let us argue again by contradiction,
and suppose that there exists $u(x)=w(x,0)=\tiw(x,0)+k$
positive eigenfunction associated to an eigenvalue $\la$
greater than $\la_{1}$.
As before, observe that Remark \ref{maxiv}
allows to choose as test function
$\psi(x,y)=v_{1}(x,y)^{2}/w(x,y)$
in the equation satisfied by $u$ and obtain
$$
\la\into\!\!\! m(x)\vfi_{1}^{2}(x)dx=\intc\!\!\! \nabla w(x,y)\left[2
\frac{v_{1}(x,y)}{w(x,y)}\nabla v_{1}(x,y)-
\left(\frac{v_{1}(x,y)}{w(x,y)}\right)^{2}\nabla w(x,y)\right]\!\!dxdy,
$$
which gives
\begin{align*}
(\la-\la_{1})\into m(x)\vfi_{1}^{2}(x)dx =-\intc {\Big |}\nabla
v_{1}(x,y)-\frac{v_{1}(x,y)}{w(x,y)}\nabla w(x,y){\Big |}^{2}dxdy
\end{align*}
and \eqref{inf} immediately implies that $\la<\la_{1}$.
\edim
\bos\label{autopos}
Notice that changing $m$ in $-m$ we can prove that there exists
a first eigenvalue $\la_{-1}<0$ with a positive eigenfunction $\vfi_{1}$
when the following condition holds
\beq\label{ipompos}
\into m(x)dx>0,
\qquad
\exists x_{0}\in \Omega, \quad \text{such that}\quad  m(x_{0}) <0.
\eeq
Moreover, arguing as in the end of Theorem \ref{primoautopos}
it is possible to prove that there are not positive eigenvalues
with positive eigenfunctions, when $m$ has positive mean.
\eos

We are finally in the position to tackle the logistic equation
\eqref{logim2}; let us start our study with the  easy
observation concerning the autonomous problem, i.e. $m\equiv1$,
contained in the following proposition.
\bpr\label{equacost}
If $m\equiv 1$ then every non-negative solution $u$  
of \eqref{logim2} is either $u=0$ or $u=1$.
\epr

\bdim
Let $u(x)=v(x,0)=\tiv(x,0)+c_{u}$ be a solution of \eqref{logim2}.
Lemma \ref{test} implies that we can choose as test function
$\psi(x,y)=v(x,y)-1$. Since $\nabla \psi=\nabla v$ we obtain
\[
0\leq \int_{\mathC}|\nabla v(x,y)|^2dxdy
=
-\la \into u(x)(u(x)-1)^{2}dx \leq 0,
\]
and the lemma follows.
\edim
Back to the nonautonomous case, the next two results provide
a priori bounds on the set of the positive solutions of \eqref{logim2}
and on the set of the parameters $\la$.
\begin{lemma}\label{limitate}
Let
$$
M=\sup_{\Omega}m^+.
$$
Then any positive solution of \eqref{logim2} satisfies
\beq\label{stimaM}
u(x)\leq M.
\eeq
In particular, if $m(x)\leq 0$ then no positive solution exists.
\end{lemma}
\bdim
Let $u(x)=v(x,0)=\tiv(x,0)+c_{u}$ be a positive solution of \eqref{logim2}.
We use Lemma \ref{test} to choose as test function
$\vfi(x,y)=(v(x,y)-M)^{+}$ and obtain
$$
\int_{{\mathcal C}\cap {\mathcal C}^{+}}\!\!\!|\nabla \tiv (x,y)|^{2}
={\la}\into
f(x,u(x))(u(x)-M)^{+}
\leq
-\la\into[(u(x)-M)^{+}]^{2}u(x),
$$
where ${\mathcal C}^{+}=\{(x,y)\in {\mathcal C}\,:\,v(x,y)
\geq M\}$ and $f(x,s)=s(m(x)-s)$.
Since the left hand side is non-negative, we obtain that $[(u(x)-M)^{+}]^{2}u(x)=0$
a.e., and the lemma follows.
\edim
\begin{corollary}\label{coro:smooth_sols}
Any nonnegative weak solution of \eqref{logim2} is $C^{1,\alpha}(\overline{\Omega})$  
and strictly positive on $\overline{\Omega}$.
\end{corollary}
\bdim
This is an easy consequence of Lemma \ref{limitate}, and Propositions
\ref{prop:reg_inv} and  \ref{maxi}.
\edim
Concerning the set of the parameters $\la$ the following necessary
condition holds.
\ble\label{lambda}
Assume \eqref{ipomneg}. Then,
if there exists a positive solution of \eqref{logim2}, then $\la>\la_{1}$.
\ele
\bdim
Let $u(x)=v(x,0)=\tiv(x,0)+c_{u}$, be a solution of equation
\eqref{eigen} and let $v_{1}=\tiv_{1}+h$ satisfying
\eqref{eqauto1} and \eqref{eqauto2}.
By Corollary \ref{coro:smooth_sols}
we can take as a test function in the equation satisfied by $v_{1}$,
$\psi(x,y)=v^{2}(x,y)/v_{1}(x,y)$ to obtain
$$
\intc \!\!\!\!\nabla v_{1}(x,y) \left[2\frac{v(x,y)}{v_{1}(x,y)}
\nabla v_{1}(x,y)-\Big(\frac{v(x,y)}{v_{1}(x,y)}\Big)^{2}
\nabla v_{1}(x,y)\right]dxdy
=\la_{1}\!\!\into\!\! m(x)u^{2}(x)dx
$$
and by using \eqref{logim2}
\begin{align*}
0&\geq -\la_{1}\into u^{3}(x)dx-\intc {\Big|}\nabla v(x,y)-
\frac{v(x,y)}{v_{1}(x,y)}\nabla v_{1}(x,y){\Big |}^{2}dxdy
\\
&= \left(\frac{\la_{1}}{\la}-1\right)\intc \nabla |v(x,y)|^{2}dxdy,
\end{align*}
providing the conclusion.
\edim

We will obtain existence results for Problem \eqref{logim2} via Bifurcation
Theory; deve\-loping this approach we have to take into account
that every solution may have a constant component that is invisible in
the differential part of the equation, then in order to make this component
appear, we will be concerned with the map $G\from\R\times\R\times X
\to Y $ where $X$ is defined in Corollary \ref{coro:smooth_L},
$Y$ is defined as
$$
Y=\left\{ (w,t) \in C^{0,\alpha}(\ov\Omega)\times\R,\; t=\into w(x)dx,
\right\}
$$
and $G$ has components $G_{1}(\la,h,\tiu)$ and $G_{2}(\la,h,\tiu))$
given by
\beq\label{defG}
\begin{cases}
G_{1}(\la,h,\tiu)=-L_{1/2}(\tiu)+\la f(x,\tiu+h), &
\medskip\\
\dys G_{2}(\la,h,\tiu)=\into G_{1}(\la,h,\tiu)dx=\into \la f(x,\tiu+h)dx ,
\end{cases}
\eeq
for $f(x,s)=s(m(x)-s)$. Let us remark that, since $\int_\Omega G_1 dx = G_2$,
we have that the elements in the range of $G$ automatically satisfy the condition 
in the definition of $Y$.
Moreover, thanks to Corollary \ref{coro:smooth_sols}, the zeroes of $G$ correspond to
solutions of Problem \eqref{logim2}.  Of course, we are interested in nontrivial solutions.
\begin{definition}\label{defi:bif_sets}\rm
We denote the \emph{sets of trivial solutions} of $G(\la,h,\tiu)=0$ as
\[
\mathT_1 := \left\{(\la,0,0):\lambda>0\right\},\qquad
\mathT_2 := \left\{(0,h,0):h>0\right\},
\]
and the \emph{set of positive solutions} as
\[
\mathS := \left\{(\la,h,\tiu):\lambda>0,\,\tiu+h>0\text{ in }\overline{\Omega}\right\}.
\]
\end{definition}
\begin{remark}\label{rem:casi_banali}\rm
We observe that if $m\leq0$ on $\overline{\Omega}$ then Lemma \ref{limitate} 
implies $\mathS=\emptyset$.
On the other hand, if $m$ is a positive constant, reasoning as in 
Lemma \ref{equacost} we infer that
$\mathS = \left\{(\la,m,0):\lambda>0\right\}$. As a consequence, 
in the following we can assume without loss
of generality that $m$ is not constant and is positive somewhere.
\end{remark}
The following local bifurcation result is concerned with the case
of negative mean of the function $m$.
\bpr\label{bifurcation1}
Let condition \eqref{ipomneg} hold and let $\la_1$ be defined as in
Theorem \ref{primoautopos}. Then $(\la_{1},0,0)$ is a bifurcation point
of positive solutions of Problem \eqref{logim2} from $\mathT_1$,
and it is the only one. Moreover, locally near such point, $\mathS$ is
a unique $C^1$ cartesian curve, parameterized by $\lambda\in(\lambda_1,\lambda_1+\delta)$,
for some $\delta>0$.
\epr
\bdim
The proof relies on classical results about the local bifurcation from
a simple eigenvalue, see for example \cite{ampr}, Chapter 5, Theorem 4.1.\\
The derivative of $G$ with respect to the pair $(h,\tiu)$
 has components
\begin{align}\label{Gderla1}
\partial_{(h,\tiu)}G_{1}(\la,h,u)[k,\tiw]&=-L_{1/2}\tiw+\la f'_{s}(x,\tiu+h)(\tiw+k),
\\
\label{Gderla2}
\partial_{(h,\tiu)}G_{2}(\la,h,u)[k,\tiw]&=\la\into f'_{s}(x,\tiu+h)(\tiw+k)dx,
\end{align}
which, evaluated at the triplet $(\la,0,0)$, gives
$$
\partial_{(h,\tiu)}G(\la,0,0)[k,\tiw]=\big(-L_{1/2}\tiw+\la
m(x)(\tiw+k),\la\into m(x)(\tiw+k)dx\big).
$$
Now, by Remark \ref{rem:fredholm}, we have that $(\la,0,0)$ can be a bifurcation point
for positive solutions only if there exists a pair $(k,\tiw)$ with $\tiw+k>0$
belonging to the kernel of the operator
$\partial_{(h,\tiu)}G(\la,0,0)$, i.e. such that
$$
\begin{cases}
L_{1/2}\tiw=\la m(x)(\tiw+k),
\medskip\\
\dys \la \into m(x)(\tiw+k)dx=0,
\end{cases}
$$
which is equivalent to say that the function $w(x,y)=\tiw(x,y)+k$
is a positive solution of
\beq\label{lineari}
\begin{cases}
\lapneu w=\la m(x)w,
\medskip\\
\dys \la\into m(x) w\,dx=0.
\end{cases}
\eeq
For this linear eigenvalue problem,
Theorem \ref{primoautopos} shows that there exists only one
positive simple eigenvalue $\la_{1}$ with a positive eigenfunction
$\vfi_{1}$ satisfying \eqref{defla1}. Decomposing $\vfi_{1}$ as
$\tilde\vfi_{1}+c_{\vfi_{1}}=\tilde\vfi_{1}+h_{1}$, we deduce that the kernel of
the operator $\partial_{(h,\tiu)}G(\la_1,0,0)$ is generated by
$(h_1,\tilde\vfi_{1})$. By virtue of Remark \ref{rem:fredholm}, this
implies that the range of the operator $\partial_{(h,\tiu)}G(\la_{1},0,0)$  is closed and that
it has codimension  one. Such range consists in the pairs
$(\tiw,t)$ such that there exists a solution
$z(x,y)=\tiz(x,y)+h$ of the problem
$$
\begin{cases}
\lapneu z=\la_{1} m(x)z+w,
\medskip\\
\dys \la_{1}\into m(x)z \,dx=t.
\end{cases}
$$
Taking as test function $\psi(x,y)=\tiv_{1}(x,y)$
in the weak formulation of the first equation
we derive that the range is given by
$$
\left\{(w,t)\in  Y\; \text{such that}\; \dys\into w(x) \vfi_{1}(x)\,dx=0\right\}.
$$
Deriving \eqref{Gderla1} and \eqref{Gderla2}  with respect to $\la$ leads to
$$
\partial_{\la}\partial_{(h,\tiu)}G(\la,h,u)[l,k,\tiw]=\big( f'_{s}(x,\tiu+h)(\tiw+k),
\into f'_{s}(x,\tiu+h)(\tiw+k)dx\big)
$$
and, denoting with $M$ the operator 
$\partial_{\la}\partial_{(h,\tiu)}G(\la,h,u)[\la_{1},0,0]$ we have that
$$
M(h_{1},\tivfi)=\big( m(x)(\tivfi+h_{1}),
\into m(x)(\tivfi+h_{1})dx\big)=\big( m(x)\vfi_{1},
\into m(x)\vfi_{1} dx\big).
$$
At this point, in order to apply the aforementioned theorem from \cite{ampr},
we only have to check that $M(h_{1},\tivfi)$ does not belong to the range of 
$\partial_{(h,\tiu)}G$,
and this occurs because
$$
\into  m(x)(\tivfi(x)+h_{1})(\tivfi(x)+h_{1})dx=\into
m(x)\vfi_{1}^{2}dx=\frac1{\la_{1}}\intc |\nabla \vfi_{1}|^{2}= 1.
$$
Then at $(\la_{1},0,0)$ a bifurcation occurs. Moreover,
as $f(x,s)=s(m(x)-s)$ is of class $C^{2}$ with respect to $s$,
the set of the nontrivial solution of $G(\la,\tiu,h)=0$
near $(\la_{1},0,0)$ is a unique $C^{1}$ cartesian curve,
parameterized by
\beq\label{locfun}
\la =\la_{1}+\mu(t),
\quad  h =th_{1}+\beta(\la_{1}+\mu(t),t\tivfi),
\quad \tiu =t\tivfi+\gamma(\la_{1}+\mu(t),t\tivfi)
\eeq
for $t\in (-\eps,\eps)$, $t\neq 0$. Here both $\gamma(\la_{1}+\mu(t),t\tivfi)$ and
$\beta(\la_{1}+\mu(t),t\tivfi)$ are $o(t)$ as $t\to0$, while a direct computation
shows that $\mu'(0)>0$. Thus, for sufficiently small  $t>0$, it is possible to write
$t=t(\lambda)$, and the solution $(\la,h,\tiu)$ is positive.
\edim
Coming to the case of positive mean of the function $m$, it is more convenient
to use as a bifurcation parameter $h$ instead of $\lambda$.
\bpr\label{bifurcation2}
Assume
\beq\label{mediapos}
\into m(x)dx>0,
\eeq
and let $h^*$ be defined as
\beq\label{defh}
h^{*}=\slashint_\Omega m(x)dx.
\eeq
Then $(0,h^{*},0)$ is a bifurcation point
of positive solutions of Problem \eqref{logim2} from $\mathT_2$,
and it is the only one. Moreover, locally near such point, $\mathS$ is
a unique $C^1$ cartesian curve, parameterized by $\lambda\in(0,\delta)$,
for some $\delta>0$.
\epr
\bdim
The derivative of $G$ with respect to $(\la,\tiu)$ has
components
\begin{align}\label{Gderh1}
\partial_{(\la,\tiu)}G_{1}(\la,h,\tiu)[l,\tiw] &=-L_{1/2}\tiw+\la
f'_{s}(x,\tiu+h)\tiw+
lf(x,\tiu+h)
\\
\label{Gderh2}
\partial_{(\la,\tiu)}G_{2}(\la,h,\tiu)[l,\tiw] &=\la\into f'_{s}(x,\tiu+h)\tiw dx+l\into f(x,\tiu+h)dx,
\end{align}
so that a pair $(l,\tiw)$ belongs to the kernel of
$\partial_{(\la,\tiu)}G(0,h,0)$ if and only if $(l,\tiw)$ solves the problem
\beq\label{line2}
\begin{cases}
-L_{1/2}\tiw+lf(x,h)=0,
\medskip\\
\dys l\into f(x,h)dx=0.
\end{cases}
\eeq
For $l=0$, taking into account Corollary \ref{coro:smooth_L}
we find $\tiw=0$, while for $l\neq 0$ the mean of
$ f(x,h)$ has to be zero and this, thanks to \eqref{mediapos}, 
yields the positive value for $h^{*}$ given by \eqref{defh}.
With this choice of $h^{*}$, for any $l$ there exists a unique solution 
of the first equation. Denoting with  $\tiz^{*}$ the one corresponding to 
$l=1$, we obtain that the kernel of $\partial_{(\la,\tiu)}G(0,h,0)$ is
the one dimensional space generated by the pair $(1,\tiz^{*})$.
\\
On the other hand, pair $(w,t)$ belongs to the range of
$\partial_{(\la,\tiu)}G(0,h^{*},0)$ if and only if there exists a solution $(\tiv,l)$ of the problem
$$
\begin{cases}
L_{1/2}\tiv=lf(x,h^{*})+w,
\medskip\\
\dys l\into f(x,h^{*})dx=t.
\end{cases}
$$
Since the function $f(x,h^{*})$ has zero mean,  $t$ has to be zero
and the range is given by the set
$\left\{(w,t)\in  Y ,\,\text{such that}\; t=0\right\}$
which is closed and of codimension one. De\-ri\-ving \eqref{Gderh1} and \eqref{Gderh2}
with respect to $h$ leads to
\begin{align*}
\partial_{h}\partial_{(\la,\tiu)}G_{1}(\la,h,u)[l,\tiw] &=\la
f^{''}_{s}(x,\tiu+h)\tiw+lf'_{s}(x,\tiu+h),
\\
\partial_{h}\partial_{(\la,\tiu)}G_{2}(\la,h,u)[l,\tiw] &
=\la\into f^{''}_{s}(x,\tiu+h)\tiw dx+l\into f'_{s}(x,\tiu+h)dx.
\end{align*}
This time we obtain the operator
$N=\partial_{h}\partial_{(\la,\tiu)}G(0,h^{*},0)$, which computed on
$(1,\tiz^{*})$ gives
$$
N(1,\tiz^{*})=\Big(f'_{s}(x,h^{*}), \into f'_{s}(x,h^{*}) dx \Big)
$$
If the second component of $N(1,\tiz^{*})\neq 0$ then
$N(1,\tiz^{*})$ does not belong to the range of
$\partial_{(\la,\tiu)}G(0,h^{*},0)$ implying that
bifurcation occurs in this case too; and this is true
since \eqref{defh} yields
$$
N(1,\tiz^{*})=\Big(m(x)-2\slashint m(x)dx, -\into m(x)dx\Big).
$$
As before, using \cite{ampr}, Chapter 5, Theorem 4.1, we deduce the existence of a
cartesian curve in a neighborhood of $(0,h^{*},0)$ with representation
$$
h =h^{*}+\nu(t),
\quad
\la =t+\alpha(h^{*}+\nu(t),tz^{*}),
\quad \tiu =tz^{*}+\gamma(h^{*}+\nu(t),tz^{*}),
$$
for $t\in (-\eps,\eps)$, $t\neq 0$. Here both $\alpha(h^{*}+\nu(t),tz^{*})$ and
$\gamma(h^{*}+\nu(t),tz^{*})$ are $o(t)$ as $t\to0$, while $\nu(0)=0$.
Since $h^*$ is positive and also $\lambda$ is positive for $t$ positive and small,
the proposition easily follows.
\edim
\bos
We stress the fact that both in Proposition \ref{bifurcation1} and in Proposition 
\ref{bifurcation2} we can locally parameterize $\mathS$ with respect to $\lambda$, 
even though in the latter the bifurcation
parameter is $h$.
\eos
\bos
If \eqref{ipompos} holds we can go through
the proof of Proposition \ref{bifurcation1}
and use Remark \ref{autopos}
to obtain that $\la_{1}<0$ is bifurcation point of
positive solutions of \eqref{logim2}
with $\la<0$ and $u=\tiu+h$ nonnegative.
Moreover, as in the case of $\la_{1}>0$, Lemma \ref{lambda}
implies that the bifurcation
occurs on the right hand side of $\la_{1}$.
Finally, let us notice that, in order to show the local bifurcation from $(0,h^{*},0)$, it
is enough to assume \eqref{mediapos} and $m$ needs not to be sign-changing.
\eos
Note that, by Proposition \ref{inverse}, it is possible to reformulate the equation 
$G=0$ in terms of a identity minus compact map, see also Remark \ref{rem:fredholm}. 
Then a classical result due to Rabinowitz \cite{Rab} implies that the continuum 
bifurcating either from $(\la_{1},0,0)$ or from $(0,h^{*},0)$ is actually global. 
Here we prefer to recover this result from a stronger one: indeed we are going 
to show that the set $\mathS$ of positive solutions
is a smooth arc.
\ble\label{prolungo}
Let $ (\la_{0},\tiu_{0},h_{0})\in\mathS$. Then there exist $\mathcal{U}\in\R\times\R\times X$
neighborhood of $ (\la_{0},\tiu_{0},h_{0})$, $\delta>0$ and a $C^{1}$ map 
$\Psi\from (\la_{0}-\delta,\la_{0}+\delta)\to \R\times X$ such that
\[
\mathS\cap\mathcal{U} = \left\{(\la,\Psi_{1}(\la),\Psi_{2}(\la)):\la\in(\la_{0}-
\delta,\la_{0}+\delta)\right\}.
\]
\ele
\bdim
The conclusion will follow from the application of the Implicit Function Theorem
to the map $G(\la,h,\tiu)$ defined in \eqref{defG}. To this aim, taking into account
\eqref{Gderla1}, \eqref{Gderla2}, we want to show the invertibility of the operator
\[
\partial_{(h,\tiu)}G(\la_{0},h_{0},\tiu_{0})[t,\tiz]=
\left(
-L_{1/2} \tiz-\la_{0}(m-2u_{0})(\tiz+t),
\la_0\into (m-2u_{0})(\tiz+t)
\right).
\]
We claim that, for such operator, the Fredholm Alternative holds. Reasoning as in
Remark \ref{rem:fredholm}, to obtain the claim it is enough to show that
$\into (m(x)-2u_{0}) \neq 0$. But this can be easily obtained by testing the equation
for $u_0$ with $1/v_0$, where as usual 
$u_0(x)=\tilde u_0(x) + h_0 = \tilde v(x,0) + h_0 = v(x,0)$:
\[
\int_\Omega (m(x)-2u_{0})dx < \into (m(x)-u_{0})dx = - \int_{\mathC}
\left|\frac{\nabla v_0}{v_0}\right|^2dxdy.
\]
Once the Fredholm Alternative is established, we have that
$\partial_{(h,\tiu)}G(\la_{0},h_{0},\tiu_{0})$ is invertible if and only
if its kernel is trivial. In turn, $(t,\tiz)$ belongs to the kernel
if and only if $z=\tiz+t$ solves the problem
$$
\begin{cases}
\dys \lapneu z=\la_{0}(m(x)-2u_{0})z,
\medskip\\
\dys \la_0\into (m(x)-2u_{0})z=0.
\end{cases}
$$
Taking $\psi(x,y)=w^{2}(x,y)/v_{0}(x,y)$ as test function
in the equation satisfied by $v_{0}$, where $w$ is the harmonic extension
of $z$, we obtain
$$
\intc \!\!\!\nabla v_{0}(x,y)\left[2\frac{w}{v_{0}}\nabla w(x,y)
-\left(\frac{w}{v_{0}}\right)^{2}\nabla v_{0}(x,y)\right]dxdy=
\la_{0}\into \!\!\!\!\left(m(x)-u_{0}(x)\right)z^{2}(x)dx.
$$
Then we can test the equation for $w$ with $w$ itself, and subtract it from the equation above.
We obtain
$$
0\leq-\into \Big|\nabla w(x,y)-\frac{w(x,y)}{v_{0}(x,y)}\nabla v_{0}(x,y)\Big|^{2}dxdy
=\la_{0}\into u_{0}(x)z^{2}(x)dx\leq0,
$$
which implies that $z$, and then $(t,\tiz)$, must vanish.
\edim

\bte\label{global} Let $\mathS$ be as in Definition \ref{defi:bif_sets}. Then
\begin{itemize}
\item[(i)] if \eqref{ipomneg}  holds then ${\mathcal S}$ is the graph of a $C^1$ map
$\Psi\from (\lambda_{1},+\infty)\to \R\times X$, with $\Psi(\lambda_1^+)=(0,0)$;
\item[(ii)] if \eqref{mediapos}  holds then ${\mathcal S}$ is the graph of a $C^1$ map
$\Psi\from (0,+\infty)\to \R\times X$, with $\Psi(0^+)=(h^*,0)$.
\end{itemize}
\ete
\bdim
To start with, we prove that $\mathS$ contains such a graph. Let us 
assume condition \eqref{ipomneg}, and let us define
$$
\Lambda:=\sup_{\la>\la_{1}}\left\{ \exists\, \Psi\in C^1((\la_{1},\la),\R\times X), \;
\mathrm{graph}(\Psi)\subset\mathS \right\}.
$$
Proposition \ref{bifurcation1} and Lemma \ref{lambda}
imply that $\Lambda>\la_{1}$, let us suppose by contradiction that
$\Lambda<+\infty$,  and consider a cartesian curve
$\Psi\from (\la_{1},\Lambda)\to \R\times X$, defined by
$\Psi(\la)=(\Psi_{1}(\la),\Psi_{2}(\la))$, with
$(\la,\Psi_{1}(\la), \Psi_{2}(\la))\in\mathS$.
Let us consider a sequence $\la_{n}< \Lambda$ tending to  $\Lambda$
with corresponding solutions $(\la_{n},h_{n},\tiu_{n})$, where
$h_{n}=\Psi_{1}(\la_{n})$, $\tiu_{n}=\Psi_{2}(\la_{n})$,  and
$u_{n}=\tiu_{n}+h_{n}$. Moreover, let us recall that we can write
$\tiu_{n}(x)=\tiv_{n}(x,0)$ and $v_{n}(x,y)= \tiv_{n}(x,y)+h_{n}$.
Taking as test function  in \eqref{logim2} $\psi(x,y)=v_{n}(x,y)$
and applying Lemma \ref{limitate}, we immediately infer the
uniform bound
\beq\label{hstima}
\|\tiv_{n}\|_{\huno}\leq L,
\eeq
from which we deduce that $\tiu_{n}$ is uniformly bounded in
the spaces $L^{p}(\Omega)$ with $1\leq p\leq 2N/(N-1)$.
Since $G_{2}(\la_{n},h_{n},\tiu_{n})=0$, where $G_{2}$ is defined in
\eqref{defG}, and  as $\tiu_{n}$ has zero mean, $h_{n}$ has to be 
positive and we obtain
$$
h_{n}^{2}|\Omega|
\leq h_{n}\left[\into m(x)dx-2\into \tiu_{n}(x)dx\right] +L
\leq cL+h_{n}M|\Omega|,
$$
for $c$ positive constant and $M$ defined in Lemma \ref{limitate}.
Hence, also $h_{n}$ is bounded and there exists $h\geq 0$ such that,
up to subsequences, $h_{n }\to h$, and $u_{n}=\tiu_{n}+h_{n}\to \tiu+h\geq 0$.
From Proposition \ref{maxi} we have two possibilities, either $u>0$ or
$u\equiv 0$. In the first case we have obtained a positive solution of
\eqref{logim2} with $\la=\Lambda$ and Lemma \ref{prolungo}
provides a contradiction with the definition of $\Lambda$. In the second
case, $\Lambda$ turns out to be a local bifurcation point for positive
solutions, but then  Proposition \ref{bifurcation1}  implies
that $\Lambda=\la_{1}$ which is again a
contradiction, showing that $\Lambda=+\infty$.
\\
When \eqref{mediapos} is assumed we define
$$
\Lambda:=\sup_{\la>0}\left\{ \exists\, \Psi\in C^1((0,\la),\R\times X), \;
\mathrm{graph}(\Psi)\subset\mathS \right\}.
$$
Then $\Lambda>0$ by Proposition \ref{bifurcation2}, and arguing as above
we obtain that also in this case $\Lambda=+\infty$.
\\
Finally, we are left to show that $\mathS\setminus\mathrm{graph}(\Psi)$ is empty. 
We prove it assuming \eqref{ipomneg}, when \eqref{mediapos} holds the same 
conclusion can be obtained with minor changes. Let us argue by contradiction 
and suppose that there exists $\la^{*}$
with distinct positive solutions $(\la^*,h_{1},\tiu_{1})$
and $(\la^{*},h_{2},\tiu_{2})$. Arguing as above, it is possible to
see that $(\la^*,h_{1},\tiu_{1})$  and $(\la^{*},h_{2},\tiu_{2})$
belong respectively to global branches ${\mathcal S_{1}}$
and ${\mathcal S_{2}}$ of positive solutions that can be parameterized
by cartesian curves $\Psi_{1}, \Psi_{2}
\from [\la_{1},+\infty)\to \R\times X$. Notice that
${\mathcal S_{1}}\cap {\mathcal S_{2}}=\emptyset$ and
neither ${\mathcal S_{1}}$ nor ${\mathcal S_{2}}$ may have
turning points, otherwise  Lemma \ref{prolungo} would be contradicted.
As a consequence $\la_{1}$ is a multiple bifurcation point
of positive solutions, but this is in contradiction with the local representation
provided in \eqref{locfun}.
\edim
\begin{corollary}
There exists exactly one positive solution $u(x)=\tiu(x)+h$ associa\-ted to 
any $\la>\la_{1}$ when \eqref{ipomneg} holds,
and there exists exactly one positive solution
$\tiu+h$ for every $\la>0$ when \eqref{mediapos} holds.
\end{corollary}
\bdim
This is an evident consequence of Theorem \ref{global}.
\edim
Taking into account Remark \ref{rem:casi_banali} we have that the only case
left uncover by Theorem \ref{global} is when $m$ has zero mean but it is not identically zero.
Notice that in such case the candidate bifurcation point is the origin, but it is not
possible to argue as in the previous results, as
the mixed derivatives $\partial_{\la}\partial_{(h,\tiu)}G(0,0,0)$,
$\partial_{h}\partial_{(\la,\tiu)}G(0,0,0)$  are now  both trivial.
Nevertheless, we can still prove the existence of
a solution for every $\la>0$ arguing by approximation.
\bte\label{media zero}
Assume that $m$ is a Lipschitz function not identically zero and satisfying
$$
\into m(x)dx=0.
$$
Then ${\mathcal S}$  is the graph of a $C^1$ map
$\Psi\from (0,+\infty)\to \R\times X$, with $\Psi(0^+)=(0,0)$.
\ete
\bdim
Let us choose $n_{0}>1$ such that for every  $n>n_{0}$  the weight
\begin{equation*}\label{mn}
m_{n}(x)=m(x)-\frac1n
\end{equation*}
satisfies hypothesis \eqref{ipomneg}. Letting $\mathcal{M}_n$ be defined
as in \eqref{defM}, with $m_{n}$ in the place of $m$, Theorem \ref{primoautopos}
yields the existence of a first positive eigenvalue
\[
\la_{1,n} = \inf_{v(x,0)\in\mathcal{M}_n}\int_{\mathC}|\nabla v(x,y)|^2dxdy
\]
associated to the weight $m_{n}(x)$. Let us define
\[
u_n(x) := p_n m(x) + q_n,\qquad\text{where }p_n:=\frac{2}{\sqrt{n}\int_\Omega m^2}.
\]
We claim that, when $n$ is sufficiently large, $q_n$ can be chosen in such a way that
$u_n\in\mathcal{M}_n$, that is,
\[
\int_\Omega m_n(x) (p_n m(x) + q_n)^2dx=1.
\]
Indeed, since $\into m^2 = \into mm_n$, by direct calculations the above 
equation can be rewritten as
\[
q_n^2-4\sqrt{n}\,q_n + n - \frac{4\into m^2m_n}{\left(\into m^2\right)^2}=0,
\]
which solvability is equivalent to the condition
\[
3n +  \frac{4\into m^2m_n}{\left(\into m^2\right)^2} \geq0,
\]
trivially satisfied for $n$ large. With this choice of $q_n$ we have that, denoting
with $\hat m$ the Neumann harmonic extension of $m$ and writing 
$v_n(x,y)= p_n \hat m(x,y) + q_n$,
it holds $v_n(x,0)=u_n(x)$. This implies
\[
\la_{1,n} \leq \int_{\mathC}|\nabla v_n(x,y)|^2dxdy 
= p_n^2 \int_\mathC |\nabla \hat m(x,y)|^2dxdy,
\]
yielding $\lambda_{1,n}\to 0$ as $n\to\infty$.
\\
Now, Theorem \ref{global} provides a sequence
of $C^{1}$ functions $\Psi_{n}\from [\la_{1,n},+\infty)\to \R\times X$
with $\Psi_{n}(\la)=(h_{n},\tiu_{n})$  positive solution of \eqref{logim2}
with weight $m_{n}(x)$.
Let us fix $0<\delta<\Lambda$ and $n_{1}>n_{0}$ such that $\la_{1,n}<\delta$ for every
$n\geq n_{1}$, so that $\Psi_{n}$ is defined in $[\delta,\Lambda]$ for every $n\geq n_{1}$. 
Using Lemma \ref{limitate} and Proposition \ref{prop:reg_inv},
we obtain that $(h_{n},\tiu_{n})$ is uniformly bounded in $\R\times X$, so that up
to a subsequence $(h_{n},\tiu_{n})$ converges in $ \R\times\huno$ to
a pair $(h,\tiu)$ solution of \eqref{logim2}; moreover, the same a priori
bounds implies that $\Psi_{n}$ satisfies the hypotheses of Ascoli-Arzel\`a
Theorem in the closed, bounded interval $[\delta,\Lambda]$,
yielding the existence of a continuous function $\Psi:[\delta,\Lambda]$ such
that $\Psi_{n}$ converges to $\Psi$ uniformly and $\Psi(\la)=(h,\tiu)$.
By the arbitrariness of $\delta$ and $\Lambda$,
we have that $\Psi$ is defined in the whole interval $[0,+\infty)$, and
$\Psi(0)=(0,0)$.

The only thing left to show is that $\Psi(\la)\neq 0$ for ever $\la>0$. Let us
argue by contradiction and suppose that there exists $\la>0$ such that
$\Psi_{n}(\la)=(h_{n},\tiu_{n})\to (0,0)$. As usual, let $u_{n}=\tiu_{n}+h_{n}$ and
$v_n=\tiv_{n}+h_{n}$ be such that $v_n(x,0)=u_n(x)$. Setting
$z_{n}=u_{n}/\|\tiv_{n}\|_{\acca^{1}(\mathC)}$,  $w_{n}=v_{n}/\|\tiv_{n}\|_{\acca^{1}(\mathC)}$,
we obtain
$$
\intc \nabla w_{n}\nabla \psi
=\la\into z_{n}(m_{n}(x)-u_{n}) \psi,\qquad
\into z_{n}(m_{n}(x)-u_{n})=0,
$$
for every test function $\psi$.
Passing to the limit we obtain
$$
\intc \nabla w\nabla \psi =\la\into m(x)z \psi,\qquad
\into m(x)z=0,
$$
which is equivalent to say that the nontrivial function $z$ is a nonnegative 
eigenfunction associated to the positive eigenvalue $\la$, but as $m$ has 
zero mean value this contradicts Lemmas \ref{infty}, \ref{infty2}.
\edim

\begin{figure}
\includegraphics[width=0.48\textwidth]{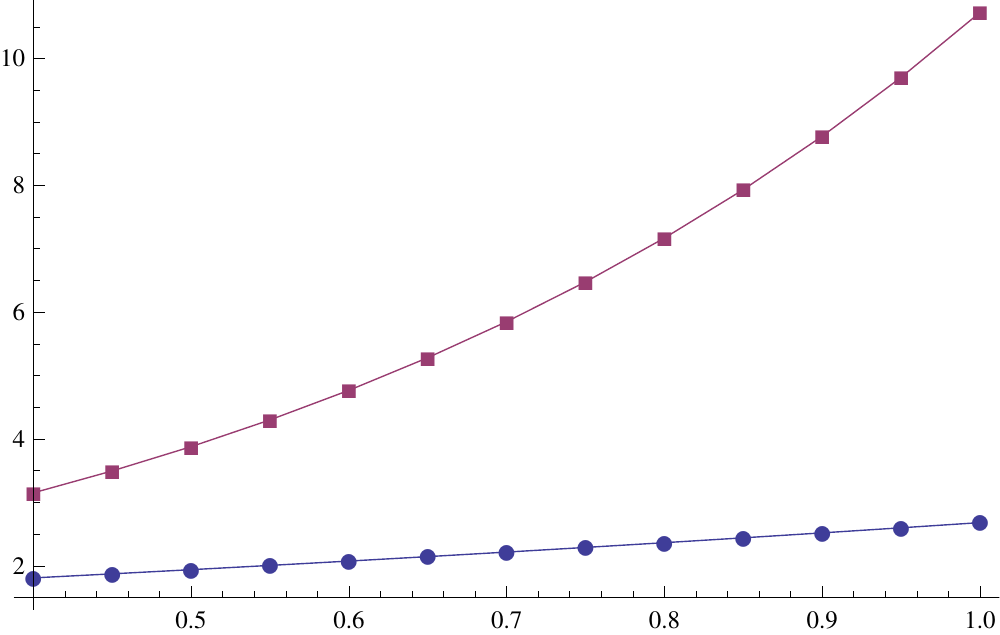}
\includegraphics[width=0.48\textwidth]{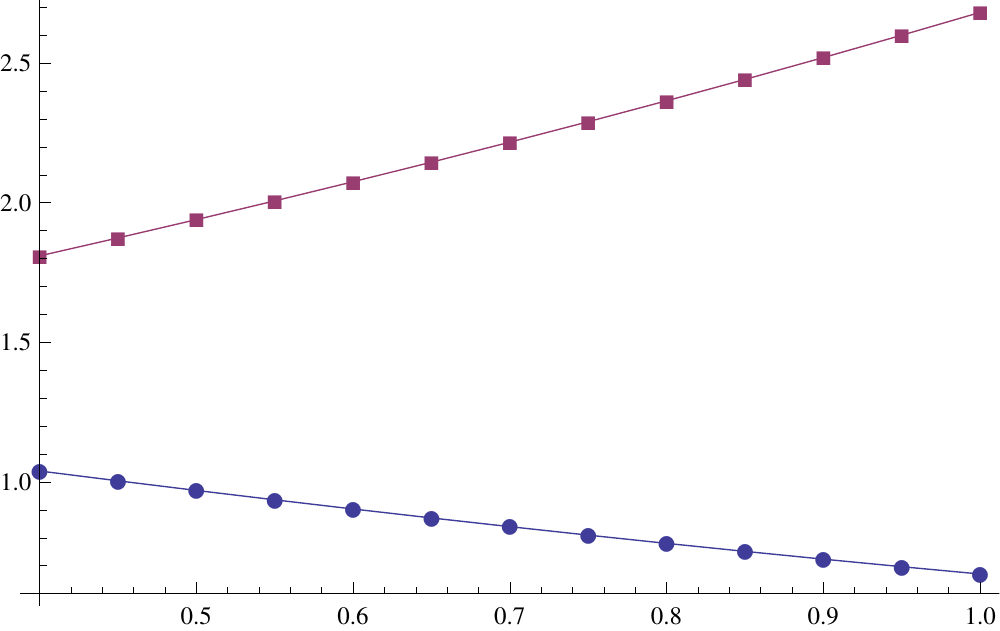}\\
\includegraphics[width=0.48\textwidth]{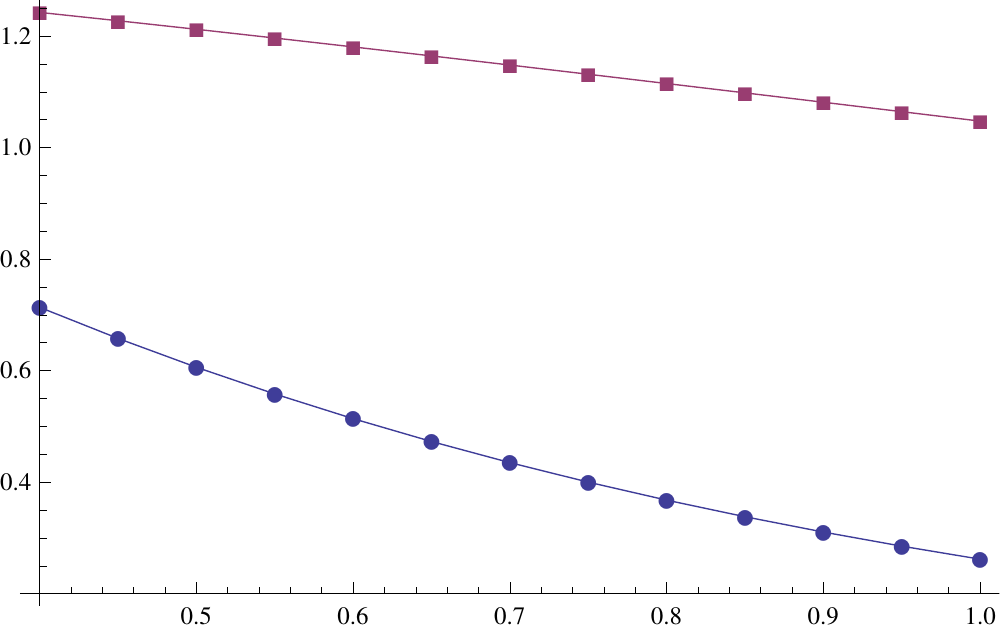}
\includegraphics[width=0.48\textwidth]{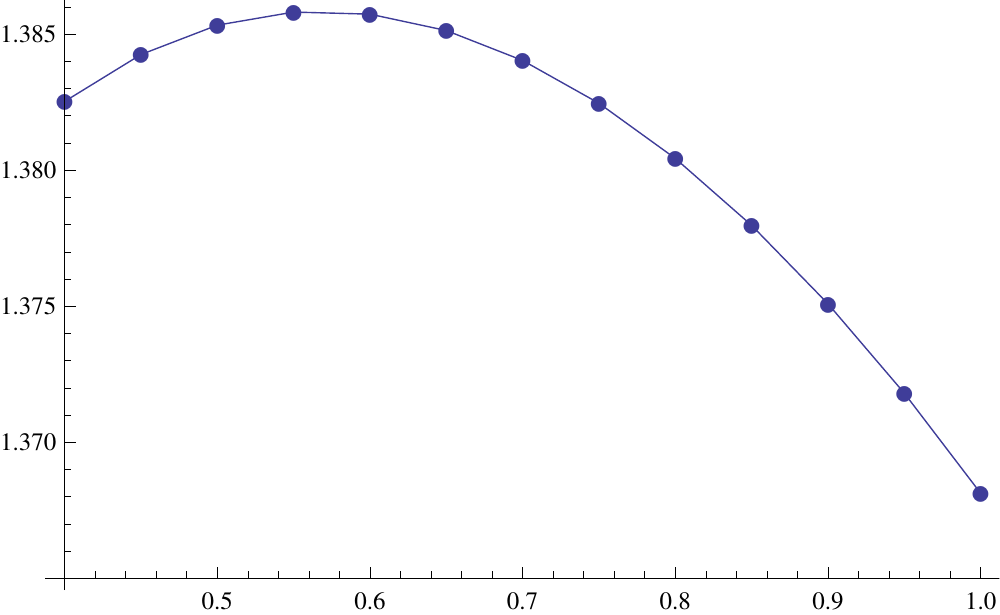}
\caption{graphs of $\Lambda(s,m_1,(0,L))$ (dots) and  $\Lambda(s,m_2,(0,L))$ (squares)
as functions of $s\in[0.4,1]$, for $L=2.5$ (top left), $L=5$ (top right), $L=8$ (bottom left).
At bottom right, the graph of $\Lambda(s,m_1,(0,L))$ for $L=3.5$.}
\label{fig:fourier}
\end{figure}
As we mentioned in the introduction, a relevant question is the one of comparing the two eigenvalues
\[
\begin{split}
\la_{1}(m,\Omega) &= \inf\left\{\int_{\mathC}|\nabla v(x,y)|^2dxdy : 
\int _\Omega m(x) v^2(x,0)dx=1,
\int _\Omega m(x) v(x,0)dx=0\right\},\\
\overline{\mu}_{1}(m,\Omega) &= \inf\left\{\int_{\Omega}|\nabla u(x)|^2dx : 
\int _\Omega m(x) u^2(x)dx=1,
\int _\Omega m(x) u(x)dx=0\right\},
\end{split}
\]
which correspond to the linearized version of \eqref{logim} and 
\eqref{deltalog}, respectively. We conclude this
section showing some simple numerical results in dimension $N=1$. Using
the usual Fourier representation with basis defined in \eqref{autof}, we have that
\[
\begin{split}
mu &= \sum_{i}\left[\int_\Omega m(x)\left(\sum_{j}u_j\phi_j(x)\right)\phi_i(x)dx \right]\phi_i\\
   &= \sum_{i}\left[\sum_{j}u_j\left(\int_\Omega m(x)\phi_j(x)\phi_i(x)dx \right)\right]\phi_i
   := \sum_{i}\left[\sum_{j} M_{ij}u_j\right]\phi_i.
\end{split}
\]
Under this point of view, solving the above minimization problems amounts 
to finding the smallest positive eigenvalue
$\Lambda(s,m,\Omega)$ of the problem
\[
\mathrm{diag}\left(\mu_i^{s}\right)_{i\geq0} \mathbf{u} = \Lambda M \mathbf{u},
\]
indeed $\la_{1}(m,\Omega)=\Lambda(1/2,m,\Omega)$ and 
$\overline{\mu}_{1}(m,\Omega)=\Lambda(1,m,\Omega)$.
In turn, such eigenvalue can be easily approximated by truncating the 
Fourier series. In Figure \ref{fig:fourier} we report these approximations 
in the cases $\Omega=(0,L)$ and
\[
m_1(x)= \cos\left(\frac{\pi}{L}x\right) - \frac12,\qquad m_2(x)
= \cos\left(\frac{2\pi}{L}x\right) - \frac12.
\]
Hence $m$ has always mean equal to $1/2$, while
\[
\mu_1(L)=\frac{\pi^2}{L^2}.
\]
We observe that in the case $L=2.5<\pi$ then $\mu_1>1$, and
 thus $\Lambda$ is increasing in $s$ for any choice of $m$
as one can trivially prove. On the other hand, when $\mu_1<1$ the 
situation is more variegated. In any case,
the eigenvalue corresponding to $m_1$ is always lower than the one 
corresponding to $m_2$, in agreement 
with the results obtained for similar weights in the case of the standard 
Laplacian in \cite{caco}. 

\bibliography{fracbib}
\bibliographystyle{abbrv}

Eugenio Montefusco,\\
Dipartimento di Matematica, \\
{\it Sa\-pien\-za} Universit\`a  di Roma,\\
p.le Aldo Moro 5, 00185 Roma, Italy.\\
E-mail address: {\tt montefusco@mat.uniroma1.it}
\medskip

Benedetta Pellacci,\\
Dipartimento di Scienze Applicate,\\
Universit\`a degli Studi di Napoli {\it Parthenope},\\
Centro Direzionale Isola C4, 80143 Napoli, Italy.\\
E-mail address: {\tt pellacci@uniparthenope.it}
\medskip

Gianmaria Verzini\\
Dipartimento di Matematica,\\
Politecnico di Milano,\\
p.za Leonardo da Vinci 32, 20133 Milano, Italy.\\
E-mail address: {\tt gianmaria.verzini@polimi.it}

\end{document}